\documentclass[11pt]{article}
\usepackage{epsfig}
\usepackage{amssymb,amsmath,amsthm,amscd}
\usepackage{latexsym}
 \usepackage{bbm,dsfont}
\usepackage{mathrsfs}

\newtheorem{thm}{Theorem}[section]
\newtheorem{prop}[thm]{Proposition}

\newtheorem{lem}[thm]{Lemma}

\theoremstyle{definition}
\newtheorem{defn}[thm]{Definition}
\newtheorem{rem}[thm]{Remark}
\newtheorem{example}[thm]{Example}

\newcommand{\be}{\begin{equation}}
\newcommand{\ee}{\end{equation}}

\newcommand{\R}{\mathbb{R}}
\newcommand{\N}{\mathbb{N}}
\newcommand{\E}{\mathbb{E}}

\newcommand{\p}{\mathbb{P}}

\def \o {{  {\mathcal{O}} }}

\def \calf {{  {\mathcal{F}} }}

\def \calb {{  {\mathcal{B}}  }}

\def \calk {{  {\mathcal{K}}  }}
\def \caly {{  {\mathcal{Y}}  }}

 \def \call {{  {\mathcal{L}}  }}

\begin{document}

\baselineskip=1 \baselineskip

\begin{center}
 {  \bf 
Well-posedness 
 of 
  Fractional
 Stochastic 
 $p$-Laplace Equations 
Driven by  Superlinear
Transport  Noise 
   }
\end{center}

\medskip

\medskip

\begin{center}
 Bixiang Wang  
\vspace{1mm}\\
Department of Mathematics, New Mexico Institute of Mining and
Technology \vspace{1mm}\\ Socorro,  NM~87801, USA \vspace{1mm}\\
Email: bwang@nmt.edu\vspace{2mm}\\
\end{center}
%



 \vspace{3mm}

\begin{abstract}  
  In this paper, we prove the
  existence
  and uniqueness  of  solutions
  of the fractional $p$-Laplace equation
   with a    polynomial drift
  of arbitrary order driven by 
  superlinear transport  noise.
  By the monotone argument, we first prove
  the existence and uniqueness of    solutions
      of an abstract stochastic differential equation
     satisfying a fully local monotonicity condition.
     We then  apply the abstract result
     to the  fractional  stochastic $p$-Laplace equation
     defined in a bounded domain.
     The main difficulty is to establish the
     tightness as well as   the uniform integrability
     of a sequence of approximate solutions
     defined by the Galerkin method.
     To obtain the necessary uniform estimates,
     we  employ the   Skorokhod-Jakubowski representation
       theorem on a topological space  instead of a metric space.
       Since the strong 
       Skorokhod  representation
       theorem is incorrect even in a complete separable metric
       space,  we 
       pass to the limit of stochastic integrals
         against a sequence
       of Wiener processes by a weak convergence argument.
     \end{abstract}



\section{Introduction}
\setcounter{equation}{0}

This paper is concerned with 
the existence and uniqueness of solutions
of the fractional $p$-Laplace equations
with a polynomial drift of any order
driven by superlinear transport noise.

Let $\o$ be a bounded domain in $\R^n$
and   $s\in (0,1)$.
Consider the  following    It\^{o}  
stochastic equation for  $x\in \o$ and 
  $t>0$:
 \begin{align}\label{lap1}
du(t,x) +(-\Delta)_p^su (t,x)
dt
= (f(t,x,u(t,x))+ h(t,x,u(t,x))) dt + B(t,x, u(t,x)) dW,
\end{align}
  with  boundary condition
  \be\label{lap2} 
  u(t,x) = 0, \quad x\in \R^n\setminus \o, \ t>0,
  \ee
  and initial condition
 \be\label{lap3}
 u( 0, x ) = u_0  (x),   \quad x\in  \R^n ,
 \ee 
  where  
 $(-\Delta)_p^s$ with $0<s<1$ and $2\leq p<\infty$ is the
 \emph{non-linear}, \emph{non-local}, \emph{fractional} $p$-Laplace operator,    
 $f: \R \times \R^n \times \R
 \to \R$ is a polynomial  nonlinearity
 with  arbitrary  growth    in its third argument,
   $h: \R \times \R^n \times \R
 \to \R$ is a    Lipschitz function 
 in its third argument.
 The diffusion coefficient $B$
 is a superlinear   function which may depend
 on   the  spatial derivatives of $u$, and
 $W$ is
 a
  cylindrical Wiener process
 in a separable  Hilbert space $U$
 defined  
 on a  complete filtered probability space
 $(\Omega, \mathcal{F}, \{ { \mathcal{F}} _t\} _{t\in \R},  \p  )$.

 Fractional partial differential equations 
     arise
     from
     anomalous diffusion 
     processes and have
     a variety of  applications in probability
      and 
   physics, see e.g., 
  \cite{abe1,  garr1,  guan2, 
    jara1, kos1}.
Such equations have been
extensively studied in
the literature,
see, e.g.,   
   \cite{abe1, bra1,  caff1,  dine1,  gal1, garr1, guan2, 
   hau1,  jara1, kos1,  luh1,   ros1,  ser2, war1, war2, val1}
for deterministic equations
and    
     \cite{gu1, luh2, ngu1,  wan1, wan2,   wanr2}
     for stochastic equations with linearly growing
     noise.
     In the present paper, we will
     apply the monotone method to
     investigate the existence
     and uniqueness of solutions
     of \eqref{lap1}-\eqref{lap3}
     when  $f$ 
     has a polynomial growth 
     of  an arbitrary  order $q\ge 2$ and $B$
     is a superlinear transport
     noise.
     
 Let
     $H
    =\{ u\in L^2(\R^n):
    u=0 \text{ a e. on } \R^n \setminus \o
    \}$, 
    $V_1 =\{ u\in W^{s,p} (\R^n):
    u=0 \text{ a e. on } \R^n \setminus \o
    \}$ , 
     $V_2 =\{ u\in L^q  (\R^n):
    u=0 \text{ a e. on } \R^n \setminus \o
    \}$, 
    and
    $\call_2(U,H)$ be the
   space of Hilbert-Schmidt
   operators from $U$ to $H$.
    The spaces  $V_1$ and $V_2$
    are used to deal with 
    the $p$-Laplace 
    operator $(-\Delta)^s_p$
    and the    polynomial nonlinearity
    $f$, respectively.
 Denote  
   the dual spaces of
    $H$, $V_1$ and $V_2$ by
    $H^*$, $V_1^*$
    and $V_2^*$, respectively.
     Let $A_1(t) = - (-\Delta)^s_p $,
    $A_2(t)= f(t, \cdot, \cdot)$
    and
    $A_3(t)= h(t, \cdot, \cdot)$.
    Then system \eqref{lap1}-\eqref{lap3}
    can be recast into the form:
    \be\label{lap4}
    du (t) = \sum_{j=1}^3 A_j (t) dt
    + B(t,\cdot, u(t)) dW \ \text{ in }\ (V_1\bigcap
    V_2)^*,
    \ee
    with $u(0) =u_0\in H$, where
    $  (V_1\bigcap
    V_2)^*$ is the dual space
    of  $ V_1\bigcap
    V_2$.
    
    In the next section,
    we actually study a more
    general stochastic equation
   that includes  \eqref{lap4}
   as a special case.
    We will assume the nonlinear
    terms in \eqref{lap4} satisfy
    the fully local monotonicity
    condition in the sense that
    for all
     $t\in [0,T]$ and  $u,v\in 
    V_1 \bigcap V_2$,
  $$
    2\sum_{j=1}^3
    (A_j(t, u)- A_j(t,v),
    u-v)_{(V_j^*, V_j)}
    +\| B(t,u)-B(t,v)\|
    _{\call_2(U,H)}^2
$$
   \be\label{lap5}
    \le  (g(t)
    + \varphi (u)
    +\psi (v))
    \| u-v\|_H^2,
\ee
    where $g\in L^1([0,T])$,
   $\varphi$
   and $\psi$ are functionals
   defined on 
   $V_1 \bigcap V_2$
   satisfying certain
   growth conditions
   (see {\bf (H2)}) in the next section
    for more
   details).
    
    We remark that the  well-posedness of
    \eqref{lap4} was already established in
    \cite{wanr1} under the condition  
    \be\label{lap6}
    \text{either} \ \ \varphi (u) \equiv 0
    \ \ \text{or} \  \ \psi (v) \equiv 0.
    \ee
    However, for a general 
    nonlinearity $f$ of polynomial type in \eqref{lap4},
  the  condition \eqref{lap6} is not satisfied, and hence
  the result of \cite{wanr1}
  does not apply in such a  case.
  The goal of this paper is to prove the
  existence and uniqueness of solutions
  of
  \eqref{lap4} when both $\varphi$ and $\psi$
  are nonzero.
  
  Note that
  if the growth rate  $q$  of  $f$ is restricted
    such that $V_1$ is embedded into
    $V_2$, then
    the stochastic equation \eqref{lap4}
    reduces to the following one defined in
    $V_1^*$ instead of  $(V_1\bigcap
    V_2)^*$:
     \be\label{lap7}
    du (t) = \sum_{j=1}^3 A_j (t) dt
    + B(t,\cdot, u(t)) dW \ \text{ in }\  V_1^*.
    \ee
   Because the growth rate of   $f$
is   arbitrary  in this paper,
the space $V_1$ is not necessarily embedded
into $V_2$, which means that
  we must  use  
  $V_1$ as well as  $V_2$, and study the
  stochastic equation \eqref{lap4}
  rather than \eqref{lap7}.

  The well-posedness of the stochastic equation
  \eqref{lap7}
  has been extensively studied in the literature
  by the monotone argument,
 see, e.g., 
 \cite{kry1, liu1, liu2, liu3, par1, zha1}
 and the references therein.
 In particular,  
  the existence and uniqueness of solutions
  of \eqref{lap7} was established in
  \cite{liu1, liu2, liu3}
   under   the  condition \eqref{lap6}.
  When both $\varphi$ and $\psi$
  in \eqref{lap5}
  are nonzero, the well-posedness of
  \eqref{lap7} was established first
  in \cite{liu4} for additive noise,
  and very recently in \cite{roc1} for
  general nonlinear noise.

  The result of \cite{roc1} solved a longstanding
  open problem on the existence and uniqueness
  of solutions of \eqref{lap7}
  under  the fully local monotonicity condition
  \eqref{lap5}.
  However, there are a couple of gaps  and
  limitations on the results of \cite{roc1}
  as listed below and we will solve these problems
  in the present paper:

    (i).    The
       authors of \cite{roc1} used the
      {  strong}   Skorokhod  representation
       theorem  to pass to the limit of a sequence
       of approximate solutions.
       However,
       the  { strong}   Skorokhod  representation
       theorem   
        does not hold true  as  proved very recently
        by 
        Ondrejat and  Seidler in \cite{ond1}
         (see also \cite{pec1}).        
         In this paper, under 
         additional assumptions
         on the continuity of noise 
         coefficients
         with respect to an appropriate
         topology, we will  deal with  the problem
        by applying the standard
        Skorokhod  representation
       theorem.

      This not only makes it more complicated
      to pass to the limit of 
      stochastic integrals related to a sequence of
      approximate solutions,
      but also requires more restricted conditions
      on the coefficients of noise.
      In this paper, 
      by  using the idea of weak convergence,       
     under additional conditions on
      the continuity of    the diffusion term $B$,
      we first prove
      $v^* B(\cdot, \widetilde{u}_{n})$
      converges to
      $v^* B(\cdot, \widetilde{u})$
      in $L^2([0,T]\times \widetilde{\Omega},
     \call_2(U,\R))$ for every
      $v\in V_1 \bigcap V_2$, where $v^*$ is the element
      in $H^*$ identified with $v\in H$
      by Riesz's representation theorem.
      We then prove the convergence
      of the stochastic integrals
     $ \int_0^\cdot v^* B(\cdot, \widetilde{u}_{n}) d
      \widetilde{W}_n$
      in $L^2(0,T; \R)$ in probability
       (see  the proof of \eqref{limeq}).

  (ii).  The proof  of \cite{roc1}
  for 
  the tightness  of solutions in  $L^r(0.T; H)$
  is incomplete. We will  solve  the problem
  by  using an appropriate embedding 
  theorem.

  By   \cite[Theorem 5]{sim1} we know that
  a subset $\caly$ of 
  $L^r(0.T; V_1)$ is  relatively compact
  in
  $L^r(0.T; H)$ 
   if 
 $\caly$ is bounded in $L^r(0,T;V_1)$
 and
 \be\label{lap10}
 \lim_{\delta \to 0^+}
 \sup_{g\in \caly}
 \int_0^{T-\delta}
 \| g(t+\delta) -g(t)\|^r_H dt =0.
 \ee
     It is clear that
     \eqref{lap10} is equivalent to: 
     \be\label{lap11}
 \lim_{k\to \infty }
 \sup_{g\in \caly}
 \int_0^{T-\delta_k}
 \| g(t+\delta_k) -g(t)\|^r_H dt =0,
 \quad \text{for every  sequence } \ 
 \delta_k \to 0^+.
 \ee
The authors of \cite{roc1}
 only proved  that
 there exists a particular sequence
 $\delta_k \to 0^+$ such that
 \eqref{lap11} is fulfilled,   
 see \cite[(5.58),  p. 3467]{roc1},
 but  did not prove \eqref{lap11}
 holds  for  {  every}
 sequence  $\delta_k \to  0^+$.
 Because of this, 
  the relative compactness
 of the subset $\caly$ constructed in
 (5.58) in \cite{roc1} is unknown.
    
  In this paper, 
  we will    find a
  new approach to establish  
     the tightness of  solutions
      in  $L^r(0.T; H)$,
      which is a key for
      proving the main result of the paper.
      To that end, we  carefully
      choose a separable Hilbert space 
      $ \mathscr{H}$ as in \cite{brz1}
       such that: 
       
        (1) $ \mathscr{H}$ is densely
      and continuously embedded into
      $V_1 \bigcap V_2$;
      and
       
        (2)  $ \mathscr{H}$ has   an
        orthogonal  basis which also 
          forms an
      orthonormal basis of 
     $H$.\\
    Then by using the uniform estimates of
    solutions, we construct a set 
    $\calk$   such that
    $\calk$ is bounded   both in
     $L^r
    (0,T; V_1\bigcap V_2)$
    and  in $W^{\sigma, 2}(0,T; \mathscr{H}^*)$
  with  $ \sigma \in  (  {\frac 12}-{\frac 1r}, \
  {\frac 12})$.
    Since the embedding
    from  $V_1\bigcap V_2$ to $H$ is compact,
    by \cite[Corollary 5]{sim1} we conclude the
    set $\calk$ is relatively compact
    in $L^r
    (0,T; H)$,  from which 
    we are able to obtain the tightness of the
    sequence of approximate solutions
    in 
     $L^r
    (0,T; H)$ (see Lemma \ref{tig1} for more details).

  (iii).  The authors of \cite{roc1}
  used the 
   Skorokhod  representation
       theorem in a   metric  space; 
       while in  
      this paper, 
       we will use   the 
        Skorokhod-Jakubowski representation
       theorem in  a  non-metric space,
       which will  not only 
       provide stronger uniform estimates of solutions
       but also weaken the assumptions on the noise
       coefficients.

  The Vitali theorem and the Lebesgue
  dominated convergence will be frequently
  used to prove the limit of a  sequence
  of approximate solutions is a solution of
  the original  equation by the monotone method.
  To apply the Vitali theorem, we need to
  establish the uniform integrablity 
   of approximate
  solutions in certain Banach spaces;
  while to apply the Lebesgue dominated
  convergence theorem, we need to
  control the approximate sequence by a Lebesgue
  integrable function.
  In both cases, we have to 
  use the     Skorokhod representation theorem
  in   the space
  $ L_{w^{*}}^{\infty}(0, T ; H)
  \bigcap L^r(0,T; H)$, 
  where
   $ L_{w^{*}}^{\infty}(0, T ; H)$
     is the space
   $ L  ^{\infty}(0, T ; H)$
  endowed with the
  weak-* topology.

  Note that the space $ L_{w^{*}}^{\infty}(0, T ; H)$
  is not  metrizable, and 
  hence 
    we cannot apply the classical
    Skorokhod representation theorem in a metric
    space.
     Instead,   we must employ the
     Skorokhod-Jakubowski
      representation theorem in 
      a topological space.

    It is worth mentioning that
    the tightness of the sequence of
    approximate solutions in
    $ L_{w^{*}}^{\infty}(0, T ; H)$ implies  
    the existence of  
  a  sequence of random variables
  defined in a new probability space, 
  which is pathwise bounded
    in $ L ^{\infty}(0, T ; H)$.
    This fact
  is frequently used to verify the uniform
  integrability of functions 
  when taking the limit of   approximate
  solutions, see Remark \ref{2.8a}
  for more details.

  (iv).  The growth
  rate of polynomial nonlinearity
  in \cite{roc1} is  restricted,
  which cannot be  arbitrarily large.  
  In the present paper, 
  we will remove this restriction,
  and  deal with 
   polynomial nonlinearity of arbitrary order.

   The results of this paper apply to
   a wide class of stochastic partial
   differential equations with 
   polynomial nonlinearity of arbitrary order,
     which include the $p$-Laplace equation
  with polynomial nonlinearity and the
  three-dimensional
  tamed Navier-Stokes equation
  with polynomial nonlinearity.
 In addition, 
  all the stochastic equations
  with transport noise considered in
  \cite{roc1} fit into 
  the framework of this paper,
  which include the 
  two-dimensional Navier-Stokes equations,
  Allen-Cahn equations, Cahn-Hilliard equations,
  Allen-Cahn-Navier-Stokes equations
  and many others.

      We mention that when the noise
      has a subcritical growth
      and does not depend on
      the spatial derivatives 
      of solutions, then the existence
      of martingale solutions
      of the fractional reaction-diffusion equation
      was established in \cite{wan3}.
      In this paper, we deal with
      the fractional $p$-Laplace equation
      when the noise
      has a critical growth
     which may 
       depend on
      the spatial derivatives  of solutions.

  In the next section, 
  by the monotone argument and
  the Galerkin finite-dimensional approximation,
  we prove the  existence and uniqueness
  of solutions of an abstract stochastic equation
 which satisfies the fully local monotonicity condition
 (see Theorem \ref{main}).
 This  result  
  can be used to deal with
 polynomial nonlinearity of any order, and is
   an extension of \cite{roc1}.
   In Section 3,
   we  apply Theorem \ref{main}
   to deal with the fractional $p$-Laplace
   equation \eqref{lap1}
   driven by superlinear noise. We first
   establish the well-posedness of \eqref{lap1}-\eqref{lap3}
   when the nonlinear term $f$ satisfies a general
   monotonicity condition with polynomial growth 
   of any order (see Theorem \ref{ma1pla}).
   We then  improve  the result 
   when $f$ satisfies a stronger  monotonicity condition
   (see Theorem \ref{ma2pla}).
   Finally, we deal with the
   standard fractional Laplace equation
   with $p=2$ driven by transport noise
   (see Theorem \ref{ma3pla}).

After the first version of    this paper
was posed on arXiv,  the authors of \cite{roc1}
informed the present author that they were aware of the gaps
in their paper and would
put an  erratum to their  original paper to fix the problems.

\section{Well-posedness  of an abstract equation}
\setcounter{equation}{0}

           In this section, we prove the existence
           and uniqueness 
           of   solutions
         to  an abstract  stochastic equation
         which satisfies the fully local
         monotonicity condition.
           We first  define a sequence of
           approximate solutions
           by   the Galerkin method,
           and   
              derive the uniform
           estimates of   these solutions.
           We then
         prove  the tightness
           of distributions of approximate  solutions,
           and apply the 
           Skorokhod-Jakubowski
      representation theorem in 
     topological spaces to obtain 
     the existence of solutions  
           by 
           the monotone argument
           as in \cite{roc1}.

\subsection{Assumptions and main results}

Throughout this section,
we assume 
     $H$ is  a separable Hilbert space with norm
      $\| \cdot \|_H$ and inner product $(\cdot, \cdot)_H$,
      and  
      $V_j$ is a separable  reflexive Banach space
      with norm $\| \cdot \|_{V_j}$ 
      for      $j=1,2\cdots J$.
       Let $H^*$ and  $V_j^*$   be the dual spaces of
      $H$ and  $V_j$ 
      with  
      duality pairings
      $(\cdot, \cdot)_{(H^*, H)}$,
      and $(\cdot, \cdot)_{(V_j^*, V_j)}$,
        respectively.

      We further  assume 
       $H$ and $V_j$   are continuously embedded in a Hausdorff topological vector space  $\mathbb{X}$.
 Then 
      $V=\bigcap_{j=1}^J V_j$
      is well-defined which is  a separable Banach space
       with norm
      $\|\cdot \|_V
      =\sum_{j=1}^J \|\cdot \|_{V_j}$. 
       In the sequel,
      we  assume  $V$ is  also  reflexive.
      The
      duality pairing between $V$ and
      its dual space $V^*$ is written as
       $(\cdot, \cdot)_{(V^*, V)}$.
       
     Suppose 
       $V$ is densely and continuously  embedded into $H$.
       Then by 
       identifying $H$  with $H^*$
       by the Riesz  representation theorem, we have
      $$ V \subseteq H \equiv H^* \subseteq  V^*,
       $$
    which means that  for all $h\in H$ and $v\in V$,  
       $$
       (h,v)_{(V^*, V)} 
       = (h, v)_{H}.
       $$

       Note that 
          for all  $j=1,2, \cdots,  J$,
       $V$
       is continuously embedded into
       $V_j$ and $V_j^*$ is continuously embedded
       into $V^*$, and hence 
           $\sum_{j=1}^J V_j^*$ is  also
           continuously embedded into   $V^*$.
For convenience,  we denote  the 
       norm of   $\sum_{j=1}^J V_j^*$ 
       by
       $$
       \| f \|_{\sum_{j=1}^J V_j^*}
       =\inf \left  \{
       \sum_{j=1}^J \|f_j\|_{V_j^*},
       \  f=\sum_{j=1}^J f_j , \ f_j \in V_j^* \ \text{ for } \ j=1,2, \cdots,
       J
       \right \},
       \quad   \ f\in \sum_{j=1}^J V_j^*.
       $$
    Notice that
     for all  $v\in V$ and $f=\sum_{j=1}^J f_j
     $
          with $f_j\in V_j^*$,  we have
        $$
        (f, v)_{(V^*, V)}
        =\sum_{j=1}^J
        (f_j, v)_{(V_j^*, V_j)}.
        $$
      
       Let  $(W(t), t\in [0,T])$  be  a cylindrical
 Wiener process in
 a separable Hilbert space  $U$  defined on
 a 
  complete filtered probability space
     $(\Omega, \calf, (\calf_t)_{t\in [0,T]}, \p)$ 
     which
satisfies  the usual condition.
Then
$W$ takes values in 
another separable Hilbert space $U_0$
such that  
 the embedding $U \subseteq U_0$ is
Hilbert-Schmidt.
The space 
of
  all Hilbert-Schmidt operators from $U$ to $H$
  is denoted by
   $\call_2 (U, H)$  
with norm $\| \cdot \|_{\call_2(U,H)}$.

 Let $B:
 [0,T] \times V
\to \call_2(U,H)$ be  
 $(\calb ([0,T]) \times \calb(V) , \calb (
\call_2(U,H)))$-measurable, 
$A_j: [0,T] \times   V_j
\to V_j^*$ be  
$(\calb ([0,T]) \times  \calb (V_j) , \calb (V_j^*))$-measurable, 
  and define
  $A: [0,T]   \times V
\to V^*$  by 
$$
A(t,   v)
=\sum_{j=1}^J A_j (t,   v),
\quad \forall \ t\in [0,T],
\ v\in V.
$$

From now on, we assume
the following assumptions
{\bf (H1)}-{\bf (H5)} are fulfilled.

{\bf (H1)} (Hemicontinuity)
  For all  $t\in [0,T]$ 
   and
  $v_1, v_2, v_3 \in V$, the function
    $(A(t,  v_1 +\delta v_2),
   v_3)_{(V^*, V)} $ is  continuous
  in  $\delta \in \R$.

     {\bf (H2)} (Local monotonicity)  For all
$t\in [0,T]$ 
   and
  $u, v  \in V$,
  $$
  2(A(t,   u)
  - A(t,  v),
  \ u-v)_{(V^*, V)}
  + \| B(t,   u) -B(t,  v)^2_{\call_2(U,H)}
  $$
  \be
  \label{h2a}
  \le
  \left (
  g (t ) +  \varphi (u) + \psi (v)
  \right ) \| u-v\|^2_H,
\ee
  where $g \in L^1([0,T] )$, 
  and $\varphi, \psi: V \to \R$ are measurable  functions such that
  for all $u, v \in V$,
\be\label{h2b}
   |\varphi (u)|   
   \le
   \alpha_1
   + \alpha_1
  \sum_{j=1}^J
  \left (
    \|u \|_{V_j}^{q_j}
    +   \|u \|_{V_j}^{q_j} \| u\|_H^{\beta_{1,j}}
    +  \| u\|_H^{2+ \beta_{1,j}}
   \right )  ,
\ee
and 
 \be\label{h2c}
   |\psi (v)|   
   \le
   \alpha_1
   ( 1+ \| v \|_H^\alpha)
   +\alpha_1  
  \sum_{j=1}^J
  \left ( 1+    \| v\|_H^{\beta_{2,j}}
  \right )  
    \|v \|_{V_j}^{\theta_j}  ,
\ee
  where 
  $\alpha \ge 0$,
    $\alpha_1 \ge 0 $, 
   $q_j >1$, 
   $\theta_j \in [0, q_j)$, 
 $\beta_{1,j} \ge 0$ and  
  $\beta_{2,j} \ge 0$ 
  are    constants
  for   all   $j=1,2,\cdots, J.$

  {\bf (H3)} (Coercivity)
  For every $t\in [0,T]$ 
   and
  $ v \in V$,
  \be\label{h3a}
   (A(t , v),
   v)_{(V^*, V)}  
   \le
   -  
    \sum_{j=1}^J
    \gamma_{1,j}
     \| v \|^{q_j}_{V_j}
   +
   g  (t ) (1 +    \| v \|^2_H)  ,
   \ee
   where 
   $\gamma_{1,j}>0  $ are    constants
  for  all   $j=1,2,\cdots, J.$

   {\bf (H4)} (Growth of drift  terms)
  For every $j=1,2, \cdots, J$,  $t\in [0,T]$ 
   and
  $ v \in V$,
  \be\label{h4a}
   \| A_j (t,   v) \|^{\frac {q_j}{q_j -1}}_{V_j^*}
   \le   
    \alpha_2 \|v\|^{q_j}_{V_j} (1 +  \| v \|^{\beta_{1,j}}
     _H )
   +
    g  (t ) (1+   \| v \|^{2+ \beta_{1,j}} _H),
   \ee
   where 
   $\alpha_2 \ge 0$  is 
  a constant.
   
  {\bf (H5)} (Growth of diffusion term)
  For every    $t\in [0,T]$ 
   and
  $ v \in V$,
  \be\label{h5a}
   \| B(t,   v) \|^2_{\call_2(U,H)}
   \le        \sum_{j=1}^J
   \gamma_{2,j}
    \| v \|^{q_j }_{V_j}
   +
    g  (t) (1+     \| v \|^2 _H),
   \ee
   where
    $\gamma_{2,j} \ge 0  $ are    constants
     for  all   $j=1,2,\cdots, J,$
such that
  \be\label{h5b}
\max_{1\le j\le J}
\kappa_j < 2  \min  _{1\le j\le J}
  \gamma_{1,j} \gamma_{2,j}^{-1},
        \ee
        and 
          \be\label{h5ca}
   \kappa_j
   = 
   \max
   \{ 1+ \beta_{1,j}, \ 
   1+ \alpha, \ 
   1+ \beta_{2,j} 
   + 2\theta_j q_j^{-1} \}.
  \ee

 We also  assume  
for every $t\in [0,T]$,
$B (t,\cdot): V \to 
   \call_2(U,H)$   is continuous.
   In addition,   we assume that
if $u \in L^\infty(0,T; H)
\bigcap  (
\bigcap\limits_{1\le j\le J} L^{q_j}
(0,T; V_j)
  )$
  and 
  $\{u_n\}_{n=1}^\infty
 $ is a bounded sequence in
 $ L^\infty(0,T; H)
\bigcap  (
\bigcap\limits_{1\le j\le J} L^{q_j}
(0,T; V_j)
  )$
  such that
  $u_n  \to u$ in $L^1(0,T; H)$,
  then
  \be\label{h5c}
   \lim_{n\to \infty}
   v^*B(\cdot ,   u_n) =  v^*B(\cdot,  u)
   \ \text{ in }  \  L^2(0,T; \call_2 (U,\R)),
   \quad   \forall  \  v\in {V},
   \ee
   where 
   $v^*$ is the element in $H^*$
   identified with $v$ in $H$  by  
   Riesz's  representation theorem.

 We remark that the convergence
   \eqref{h5c} is weaker than the following
   convergence:
     \be\label{h5d}
   \lim_{n\to \infty}
 B (\cdot ,   u_n) =  B(\cdot,  u)
   \ \text{ in }  \  L^2(0,T; \call_2 (U,H)).
   \ee
     For convenience,  we now   set
\be\label{qj}
 {\tilde{q}} =\max_{1\le j\le J} q_j \  \ \text{and} \ \ 
{\underline{q}} =\min_ {1\le j\le J} q_j . \  \ 
 \ee

%

  By  \eqref{h5a},    we see that
$B(\cdot,v)$   has     superlinear growth 
in $v$ when  $ {q}_j>2$ for   some $j=1,2,\cdots, J$.
Moreover, the growth 
of $B(\cdot,v)$ in $v$  is critical in the sense that
the exponent $q_j$ in \eqref{h5a}
is exactly the same exponent in \eqref{h3a}.

Notice that 
 by  {\bf (H1)}, {\bf (H2)}
and {\bf (H4)},      the 
operator   $A(t,  \cdot): V\to V^*$
is 
 demicontinuous  
 for every $t\in [0,T]$
 in the sense that
if $v_n \to v$ in $V$, then 
$A(t,  v_n) \to A (t,  v)$
weakly in $V^*$
 (see,   \cite[ Lemma 2.1, p. 1252]{kry1},
    \cite[Remark 4.4.1]{liu1}).


Consider the stochastic equation:
\be\label{sde1}
dX(t)
= A(t, X(t)) dt
+B(t,  X(t)) d W(t), \quad t\in (0, T],
\ee
with initial condition:
\be\label{sde2}
X(0) =x \in H.
\ee

The solution of \eqref{sde1}-\eqref{sde2}
is understood in the following sense.

\begin{defn}\label{dsol}
A continuous  $H$-valued 
$\calf_t$-adapted stochastic process
$(X(t): t\in [0,T])$ is called a solution
of problem \eqref{sde1}-\eqref{sde2}
if
$$
X\in L^2(\Omega, L^2(0,T; H))
\bigcap L^{q_j} (\Omega, L^{q_j}(0,T; V_j)),
\quad \forall \  j=1,2, \cdots, J,
$$ 
and for all $t\in [0,T]$,
$$
X(t) = x
+ \int_0^t A(s, X(s)) ds
+\int_0^t
B(s, X(s)) dW(s) \quad \text{in} \ \ V^*,
$$
$\p$-almost surely.
 \end{defn}

The main result of this section  is stated below.
 
 \begin{thm}\label{main}
If 
  {\bf (H1)}-{\bf (H5)} are fulfilled and 
  the embedding $V \subseteq H$ is compact,
then for  every  
  $x\in H$,  \eqref{sde1}-\eqref{sde2}
has a unique solution 
$X$
 in the sense of
Definition \ref{dsol} .

In addition,  for every $p$ satisfying
\be\label{main p}
 1\le p <  {\frac 12} +
\min\limits_{1\le j \le J}
 (
 \gamma_{1,j} \gamma_{2,j}^{-1}
  ),
  \ee
  the following  unifrom estimates hold: 
\be\label{main 1}
\E \left (
\sup_{t\in [0,T]}
\|X(t)\|_H^{2p}
\right )
+
\E 
\left (
\left (
\sum_{j=1}^J \int_0^T
\| X(s)\|_{V_j}^{q_j}
ds
\right )^{p}
\right )
\le M(1 +\|x\|_H^{2p}),
\ee
where 
 $M=M(T, p)>0$ is a constant
depending only on $T$ and $p$.
 \end{thm}

\subsection{Approximate solutions and uniform
estimates}

In this section,  we consider a sequence
of approximate solutions
 to \eqref{sde1}-\eqref{sde2}
by the Galerkin method, and then derive
the uniform estimates of these solutions.

 Since $V$ is
 a separable Banach space
 and  densely embedded in  $H$,
 by \cite[Lemma C.1]{brz1}, we find that
 there exists 
 a
 separable Hilbert space
 $\mathscr{H} \subseteq V$ such that
 
 (i).   $\mathscr{H}$ is dense in $V$ and
 the embedding  $\mathcal{H} \subseteq V$
 is compact.
 
(ii).  There exists an orthonormal basis
 $\{ h_k \}_{k=1}^\infty$ of
 $H$ such that $h_k \in \mathscr{H}$
 for all $k\in \N$.
 
 (iii). The sequence 
 $\{ h_k \}_{k=1}^\infty$ is also an orthogonal basis
 of $\mathscr{H}$.

 By (i) we see that there exists a constant
 $C_{\mathscr{H}}>0$ such that
 \be\label{basis0}
 \| v \|_V
 \le C_{\mathscr{H}}
 \| v\|_{\mathscr{H}},
 \quad \forall \ v\in {\mathscr{H}}.
 \ee
 On the other hand,   
for every 
    $v\in H$,  by    (ii)  we have
\be\label{basis1}
 v=\sum_{k=1}^\infty
 (v, h_k)_H h_k \quad \text{in } \ H,
\ee
 and  for every 
    $v\in \mathscr{H}$,   by (iii) we have
\be\label{basis2}
 v=\sum_{k=1}^\infty
 (v, h_k)_H h_k \quad \text{in } \ \mathscr{H}.
\ee
Note 
that the series in \eqref{basis1} converges
in $H$,  and 
  the series in \eqref{basis2} converges
in $\mathscr{H}$.
By \eqref{basis2},
  for every $n\in \N$ and 
    $v\in \mathscr{H}$,   we have
\be\label{basis3}
 \| v\|_{\mathscr{H}}^2 
 = \|P_n v\|_{\mathscr{H}}^2 
 +\| (I-P_n) v\|_{\mathscr{H}}^2 ,
\ee
where $P_n: H \to  H_n=\text{span} \{h_1, \cdots,
  h_n\}$ is  the  orthogonal  projection.
 As  a consequence of \eqref{basis3}  we get,
   for every $n\in \N$ and 
    $v\in \mathscr{H}$, 
\be\label{basis4}
   \|P_n v\|_{\mathscr{H}} \le    \|  v\|_{\mathscr{H}}.
   \ee
   
   We now extend $P_n$ from $H$
   to $V^*$ and $\mathscr{H}^*$  by
  $$
  P_n v^*=\sum_{k=1}^n (v^*, h_k)_{(V^*, V)} h_k,
  \quad \forall \ v^*\in V^*,
  $$
  and
   $$
  P_n u^*=\sum_{k=1}^n (u^*, h_k)_{(\mathscr{H}^*, 
  \mathscr{H})} h_k,
  \quad \forall \ u^*\in \mathscr{H} ^*.
  $$

Let $\{u^0_k\}_{k=1}^\infty$ be  
an orthonormal
  basis of $U_0$, and $Q_n: U_0 \to
   \  \text{span} \{u_1^0, \cdots,
  u^0_n\}$ be the  orthogonal  projection.
  Consider the $n$-dimensional stochastic differential equation
for $Z_n \in  H_n$:
\be\label{ode1}
Z_n (t)
=P_n  x
+ \int_0^t P_n A(s, Z_n (s)) ds
+ \int_0^t P_n B(s,Z_n (s))Q_n  dW(s).
\ee

Note that 
 $A(t, \cdot) :  V\to
V^*$ is  demicontinuous
  for every $t\in [0,T]$,  and hence
$P_n A(t,  \cdot): H_n
\to H_n$ is continuous. 
Then by \eqref{h2a}-\eqref{h5a}
we infer from  \cite{kry1}   that
   \eqref{ode1} 
has a  unique   solution $Z_n$ on $[0,T]$.

The next lemma is
concerned with the uniform estimates
of  $\{Z_n\}_{n=1}^\infty$.

  \begin{lem}\label{ues1}
If  {\bf (H1)}-{\bf (H5)} are fulfilled,
then 
for every  $x\in H$ and    every $p $  
satisfying  \eqref{main p}, 
  there exists a constant $M_1= M_1(T, p)>0$ 
  such that for all $n\in \N$, 
  the solution $Z_n$ of \eqref{ode1} satisfies,
  $$
  \E \left (
  \sup_{t\in [0,T]}
  \|Z_n (t)\|^{2p}_H
  \right  )
  +
   \E \left ( \left (
  \int_0^T \sum_{j=1}^J  \| Z_n (t) \|_{V_j}^{q_j} dt
  \right )^{p} \right )
  + \E \left (  \int_0^{T }
  \|Z_n (s)\|_H^{2p-2}  \sum_{j=1}^J
  \|Z_n (s)\|_{V_j}^{q_j} ds
  \right )
  $$
  $$
  \le M_1  (1 + \|x\|^{2p}_H).
   $$
  \end{lem}

  \begin{proof}
   By \eqref{ode1}  
    and It\^{o}'s formula we get,
  $\p$-almost surely, for all $t\in [0,T]$,
  $$
  \|Z_n (t)\|_H^{2p}
  -\|P_n x\|_H^{2p}
  =
 2 p\int_0^t \|Z_n (s)\|_H^{2p-2}
 (A(s, Z_n (s)), Z_n(s))_{(V^*, V)} ds
   $$
   $$
   +  p\int_0^t \|Z_n (s)\|_H^{2p-2}
  \|P_n B(s, Z_n (s))Q_n\|^2_{\call_2(U,H)^2} ds
  $$
  $$
   +  2p\int_0^t \|Z_n (s)\|_H^{2p-2}Z_n^*(s) P_n B(s, Z_n(s)) Q_n dW(s)
   $$
\be\label{ues1 p2}
   + 2p(p-1) \int_0^t \|Z_n (s)\|_H^{2p-4}
   \|Z_n^*(s) P_n B(s, Z_n(s)) Q_n\|^2_{\call_2(U,\R)} ds,
\ee
 where $Z_n^*(s)$ is the element in $H^*$ identified
  with $Z_n(s)$
  by Riesz's representation theorem.
By \eqref{h3a} and  \eqref{h5a}  
we obtain from
\eqref{ues1 p2} that, 
$\p$-almost surely, for all $t\in [0,T]$,
  $$
  \|Z_n (t)\|_H^{2p}
  + 
    \int_0^t
  \|Z_n (s)\|_H^{2p-2}  \sum_{j=1}^J
  \left ( 2p \gamma_{1,j}
  -p(2p-1)\gamma_{2,j} \right )
  \|Z_n (s)\|_{V_j}^{q_j} ds
  $$
  $$
 \le \|P_n x \|_H^{2p}
  + p (1+2p) \int_0^t \left (  g (s) 
   \|Z_n (s)\|_H^{2p-2}+ g(s)  \|Z_n (s)\|_H^{2p}
  \right )ds
  $$
$$
   +  2p\int_0^t \|Z_n (s)\|_H^{2p-2}Z_n^*(s) P_n B(s, Z_n(s)) Q_n dW(s)
   $$
    $$
 \le \|P_n x \|_H^{2p}
 + (1+2p) \int_0^t |g(s)| ds
 +
    (2p-1) (1+2p) \int_0^t 
  | g(s)|   \|Z_n (s)\|_H^{2p}
  ds
  $$
   \be\label{ues1 p3}
   +  2p\int_0^t \|Z_n (s)\|_H^{2p-2}Z_n^*(s) P_n B(s, Z_n(s)) Q_n dW(s),
  \ee
  By \eqref{main p} we have
  \be\label{ues1 p3a}
 \gamma  = \min_{1\le j \le J}
   \{ 2 \gamma_{1,j}
  - (2p-1) \gamma_{2,j}\}>0.
  \ee
      
  Given $n\in \N$ and $R>0$, denote by
\be\label{stop1}
\tau_{n,R}
=\inf  \left \{
t\ge 0: \|Z_n(t)\|_H 
+
\int_0^t \sum_{j=1}^J \| Z_n(s)\|^{q_j}_{V_j} ds
 \ge R
\right \} \wedge T,
\ee
 with $ \inf \emptyset = +\infty$.
By \eqref{ues1 p3}-\eqref{ues1 p3a}  
we obtain, 
$\p$-almost surely, for all $t\in [0,T]$,
  $$
  \|Z_n ( t\wedge \tau_{n,R} )\|_H^{2p}
  +  p\gamma 
    \int_0^{t\wedge \tau_{n,R}}
  \|Z_n (s)\|_H^{2p-2}  \sum_{j=1}^J
   \|Z_n (s)\|_{V_j}^{q_j} ds
  $$
     $$
 \le \|P_n x \|_H^{2p}
 + (1+2p) \int_0^{t\wedge \tau_{n,R}} |g(s)| ds
 +
    (2p-1) (1+2p) \int_0^{t\wedge \tau_{n,R}} 
  | g(s)|   \|Z_n (s)\|_H^{2p}
  ds
  $$
   \be\label{ues1 p3b}
   +  2p\int_0^{t\wedge \tau_{n,R}} \|Z_n (s)\|_H^{2p-2}Z_n^*(s) P_n B(s, Z_n(s)) Q_n dW(s),
  \ee
It follows from  
 \eqref{ues1 p3b} that
  for all $t\in [0,T]$,
    $$
 \E \left ( \|Z_n (t\wedge \tau_{n,R})\|_H^{2p}
 \right )
  +   p \gamma 
  \E \left (  \int_0^{t\wedge \tau_{n,R}}
  \|Z_n (s)\|_H^{2p-2}  \sum_{j=1}^J
  \|Z_n (s)\|_{V_j}^{q_j} ds
  \right )
  $$
  \be\label{ues1 p5}
 \le  \|x \|_H^{2p}  
  + (1+2p)  \int_0^{T }
    |g (s)| ds
    + (2p-1) (1+2p)  \int_0^{t }
   |g (s)|  
    \E \left (
   \|Z_n (s\wedge \tau_{n,R} )\|_H^{2p} 
   \right ) 
  ds .
  \ee
  Apply 
  Gronwall's inequality to  \eqref{ues1 p5} 
  to obtain,
 for all $t\in [0,T]$,
  $$
 \E \left ( \|Z_n (t\wedge \tau_{n,R})\|_H^{2p}
 \right )
  +  p \gamma 
  \E \left (  \int_0^{t\wedge \tau_{n,R}}
  \|Z_n (s)\|_H^{2p-2}  \sum_{j=1}^J
  \|Z_n (s)\|_{V_j}^{q_j} ds
  \right )
  $$
\be\label{ues1 p6}
 \le
 \left (   \| x \|_H^{2p}  
 +  (1+2p)  
  \int_0^ {T }    |g   (s )  |  ds
   \right )e^{    (2p-1) (1+2p)   \int_0^{t } 
   |g (s)| 
    ds }.
  \ee
  Taking the limit of \eqref{ues1 p6}
  as $R\to \infty$,   by Fatou's lemma, we  get
 for all $t\in [0,T]$,
  $$
 \E \left ( \|Z_n (t )\|_H^{2p}
 \right )
  +    p \gamma 
  \E \left (  \int_0^{t }
  \|Z_n (s)\|_H^{2p-2}  \sum_{j=1}^J
  \|Z_n (s)\|_{V_j}^{q_j} ds
  \right )
  $$
     \be\label{ues1 p7}
      \le
 \left (   \| x \|_H^{2p}  
 +  (1+2p)  
  \int_0^ {T }    |g   (s )  |  ds
   \right )e^{    (2p-1) (1+2p)   \int_0^{T } 
   |g (s)| 
    ds }.
  \ee

 On the other hand,
   by 
 \eqref{ues1 p3b}
 we have  
 $$
  \E \left (
  \sup_{t\in [0,T]}\|Z_n (t\wedge \tau_{n,R})\|_H^{2p}
  \right )
 \le 
  \|x\|_H^{2p} 
  +   (1+2p)  \int_0^ {T } 
   |g  (s) | ds
   $$
   $$
  +  (2p-1) (1+2p)   \int_0^{T } 
 |g (s)|  \E \left (
   \|Z_n (s  )\|_H^{2p} \right )
  ds
   $$
 \be\label{ues1 p8}
   +  2p\E \left (
    \sup_{t\in [0,T]}
   \left |
   \int_0^ {t\wedge \tau_{n,R}}  \|Z_n (s)\|_H^{2p-2}Z_n^*(s) P_n B(s, Z_n(s)) Q_n dW(s)
   \right |\right ).
  \ee
  By \eqref{ues1 p7} we  obtain
   $$
    (2p-1) (1+2p)   \int_0^{T } 
 |g (s)|  \E \left (
   \|Z_n (s  )\|_H^{2p} \right )
  ds  
  $$
  \be\label{ues1, p9}
  \le  (2p-1) (1+2p)    \int_0^{T }  
   |g (s)|  ds \sup_{s\in [0,T]} \E \left (
   \|Z_n (s  )\|_H^{2p} \right )
   \le c_1 (1+ \|x\|_H^{2p} ),
\ee
where $c_1=c_1(g, p, T)>0$  is a constant.
By the BDG inequality, \eqref{h5a}
and  \eqref{ues1 p7},  we 
have 
 $$
 2p\E \left (
    \sup_{t\in [0,T]}
   \left |
   \int_0^ {t\wedge \tau_{n,R}}  \|Z_n (s)\|_H^{2p-2}Z_n^*(s) P_n B(s, Z_n(s)) Q_n dW(s)
   \right |\right )
   $$
     $$
\le c_2 \E \left ( \left (
   \int_0^ {T\wedge \tau_{n,R}}
     \|Z_n (s)\|_H^{4p-2}
     \|     B(s, Z_n(s))   \|^2_{\call_2(U, H)} ds
     \right )^{\frac 12} \right )
   $$
     $$
\le{\frac 12}  \E 
 \left ( \sup_{t\in [0,T]} \|Z_n (t\wedge \tau_{n,R} )\|_H^{2p}
   \right ) 
   +
   {\frac 12} c_2^2
   \E
\left (
   \int_0^ {T\wedge \tau_{n,R}}
     \|Z_n (s)\|_H^{2p-2}
     \|     B(s, Z_n(s))   \|^2_{\call_2(U, H)} ds
     \right )
   $$
    $$
\le{\frac 12}  \E 
 \left ( \sup_{t\in [0,T]} \|Z_n (t\wedge \tau_{n,R} )\|_H^{2p}
   \right ) 
   +
   {\frac 12} c_2^2
   \E
\left (
   \int_0^ {T }
    \left ( g(s) \|Z_n (s)\|_H^{2p-2} 
    + g(s)
     \|Z_n (s)\|_H^{2p}  \right )
     ds
     \right )
   $$
    $$
    +
   {\frac 12} c_2^2  
   \E
\left (
   \int_0^ {T }
     \|Z_n (s)\|_H^{2p-2}
     \sum_{j=1}^J \gamma_{2,j}
     \|  Z_n(s)    \|^{{q}_j}_{V_j}
       ds
     \right )
   $$
     $$
\le{\frac 12}  \E 
 \left ( \sup_{t\in [0,T]} \|Z_n (t\wedge \tau_{n,R} )\|_H^{2p}
   \right ) 
   + 
c_3 
   \int_0^ {T }
    | g(s)|
    ds
    \left (1+ 
     \sup_{s\in [0,T]} \E \left (
     \|Z_n (s)\|_H^{2p}  \right )\right )
   $$
    $$
    +
   {\frac 12} c_2^2
   \max_{1\le j \le J}
   \gamma_{2,j}
   \E
\left (
   \int_0^ {T}
     \|Z_n (s)\|_H^{2p-2}  
     \sum_{j=1}^J
     \|  Z_n(s)    \|^{q_j}_{V_j}
       ds
     \right )
   $$
 \be\label{ues1 p10}
\le{\frac 12}  \E 
 \left ( \sup_{t\in [0,T]} \|Z_n (t\wedge \tau_{n,R} )\|_H^{2p}
   \right ) 
   + 
  c_4 \left (1 + \| x \|^{2p}_H  
\right )   ,
\ee
 where    $c_4=c_4(g, T, p)>0$ is a constant.
   By  \eqref{ues1 p8}-\eqref{ues1 p10}
  we obtain 
   $$   \E 
 \left ( \sup_{t\in [0,T]} \|Z_n (t\wedge \tau_{n,R} )\|_H^{2p}
   \right ) 
   \le c_5  \left (1 + \| x \|^{2p}_H 
\right ),
$$
 where $c_5=c_5(g,T, p)>0$ is a constant, which implies 
 that 
\be\label{ues1 p11}  \E 
 \left ( \sup_{t\in [0,T]} \|Z_n (t  )\|_H^{2p}
   \right ) 
   \le c_5  \left (1 + \| x \|^{2p}_H 
\right ).
\ee

As a special case of  \eqref{ues1 p3b}
with $p=1$, we have,  
 $\p$-almost surely, 
  $$
\gamma  \int_0^{T\wedge \tau_{n,R}}
  \sum_{j=1}^J
  \|Z_n (s)\|_{V_j}^{q_j} ds
  \le \| x\|_H^{2}  
  + (1+ 2p)  \int_0^{T }  
   |g (s)| ds
   $$
   $$
   + (2p-1)
   (1+ 2p)  \int_0^{T }  
   |g (s)| 
   \|Z_n (s)\|_H^{2} 
  ds  
   +  2\int_0^ {T\wedge \tau_{n,R}}  
   Z_n^*(s) P_n B(s, Z_n(s)) Q_n dW(s),
    $$
    and hence
   $$
  \gamma^p
  \E \left (
  \left (
    \int_0^{T\wedge \tau_{n,R}}
  \sum_{j=1}^J
  \|Z_n (s)\|_{V_j}^{q_j} ds  \right )^p \right )
  $$
  $$
  \le c_6 ( 1+
   \| x \|_H^{2p}   )
   +c_6 
   \E \left (
   \left (  \int_0^{T }  
 |g (s)| 
   \|Z_n (s)\|_H^{2}  
  ds 
  \right )^p \right )
  $$
  $$
   +   4^{p-1} 2^p
   \E \left (
   \left |
   \int_0^ {T\wedge \tau_{n,R}}  
   Z_n^*(s) P_n B(s, Z_n(s)) Q_n dW(s)
   \right |^p
   \right ),
    $$
    $$
  \le 
   c_6 ( 1+
   \| x \|_H^{2p}   )
   +c_6 
    \left (  \int_0^{T}  
 |g (s)| ds  \right )^p
      \E \left (\sup_{s\in [0,T]}
   \|Z_n (s)\|_H^{2p}  
   \right )
  $$
 \be\label{ues1 p12}
   +  4^{p-1}  2^{p}
   \E \left (
   \left |
   \int_0^ {T\wedge \tau_{n,R}}  
   Z_n^*(s) P_n B(s, Z_n(s)) Q_n dW(s)
   \right |^p
   \right ),
  \ee 
  where $c_6=c_6
  (g, T, p)>0$ is a constant.
 By  the BDG inequality, \eqref{h5a}  and \eqref{ues1 p11}  we 
  obtain 
$$
     4^{p-1}  2^{p}
   \E \left (
   \left |
   \int_0^ {T\wedge \tau_{n,R}}  
   Z_n^*(s) P_n B(s, Z_n(s)) Q_n dW(s)
   \right |^p
   \right )
$$
$$
  \le c_7  
   \E \left (
   \left  (
   \int_0^ {T\wedge \tau_{n,R}}  
   \| Z_n(s)\|^2 \|  B(s, Z_n(s))  \|^2_{\call_2(U, \R)}
   ds
   \right )
   ^{\frac p2}
   \right )
   $$
   $$
  \le c_7  
   \E \left (
   \left  (
   \int_0^   {T\wedge \tau_{n,R}}     (
      \| Z_n(s)  \|^2_H
   \sum_{j=1}^J \gamma_{2,j} 
    \| Z_n(s)  \|_{V_j}^{{q}_j}
   +  
   g(s) (
   \| Z_n(s)  \|^2_H
  +
   \| Z_n(s)  \|^4_H
    )   )
   ds
   \right )
   ^{\frac p2}
   \right )
   $$
    $$
  \le  2^{\frac p2} c_7
\max_{1\le j \le  J} \gamma_{2,j}
 ^{\frac p2}
   \E \left (
   \left  (
   \int_0^  {T\wedge \tau_{n,R}}    
     \| Z_n(s)  \|^2_H
   \sum_{j=1}^J
    \| Z_n(s)  \|_{V_j}^{{q}_j} ds
     \right )
   ^{\frac p2}
   \right )
   $$
   $$
   +     2^{\frac p2}
  c_7  
   \E \left (
   \left  (
   \int_0^ {T } 
  | g(s)|  (
   \| Z_n(s)  \|^2_H
  +
   \| Z_n(s)  \|^4_H
    )   
   ds
   \right )
   ^{\frac p2}
   \right )
   $$
    $$
  \le 
  2^{\frac p2} c_7
\max_{1\le j \le  J} \gamma_{2,j}
 ^{\frac p2} 
   \E \left (
   \sup_{s\in [0,T]}
   \| Z_n(s)  \|^p _H
   \left  (
   \int_0^{T\wedge \tau_{n,R}}  
   \sum_{j=1}^J
    \| Z_n(s)  \|_{V_j}^{{q}_j} ds
     \right )
   ^{\frac p2}
   \right )
   $$
   $$
   +     2^{\frac p2}
  c_7   \left (
   \int_0^ {T } 
   |g(s)|  ds \right )^{\frac p2}
   \E \left (
   \left  ({\frac 12} +{\frac 32}
   \sup_{s\in [0,T]}
   \| Z_n(s)  \|^4_H
   \right )
   ^{\frac p2}
   \right )
   $$
     $$
  \le 
 {\frac 12} \gamma^p 
    \E \left ( 
   \left  (
   \int_0^   {T\wedge \tau_{n,R}}    
   \sum_{j=1}^J
    \| Z_n(s)  \|_{V_j}^{{q}_j} ds
     \right )
   ^p
   \right )
  + c_8 + 
  c_8
   \E \left (
   \sup_{s\in [0,T]}
   \| Z_n(s)  \|^{2p} _H
   \right )
   $$ 
   \be\label{ues1 p13}
    \le 
{\frac 12} \gamma^p 
    \E \left ( 
   \left  (
   \int_0^  {T\wedge \tau_{n,R}}    
   \sum_{j=1}^J
    \| Z_n(s)  \|_{V_j}^{q_j} ds
     \right )
   ^p
   \right )
 + c_9  \left (1 +  \| x \|^{2p}_H  
\right ),
\ee
    where $c_9=c_9(g, T, p)>0$ is a constant.
   By  \eqref{ues1 p11}-\eqref{ues1 p13}
 we obtain 
  \be\label{ues1 p14} 
   \E \left (
  \left (
    \int_0^{T\wedge \tau_{n,R}}
  \sum_{j=1}^J
  \|Z_n (s)\|_{V_j}^{q_j} ds  \right )^p \right )
    \le 
   c_{10}  \left (1 +  \| x\|^{2p}_H 
\right ),
\ee
   where $c_{10}=c_{10}(g, T, p)>0$ is a constant.
  Take the limit of   \eqref{ues1 p14}
   as $R\to \infty$ to obtain 
$$
  \E \left (
  \left (
    \int_0^{T }
  \sum_{j=1}^J
  \|Z_n (s)\|_{V_j}^{q_j} ds  \right )^p \right )
  \le 
   c_{10}  \left (1 + \| x \|^{2p}_H 
\right ),
$$
 which along with
  \eqref{ues1 p7} and
  \eqref{ues1 p11}
 concludes the proof.
    \end{proof}

   \begin{lem}\label{ues2}
   If  {\bf (H1)}-{\bf (H5)} are fulfilled, 
then 
for any     $x\in H$, 
  there exists a constant $M_2= M_2(T)>0$ 
  such that
  the solution $Z_n$ of \eqref{ode1} satisfies,
  for all $n\in \N$  and
   $j=1,\cdots, J$,
    $$ 
     \E \left (
     \int_0^T
   \| A_j (s,   Z_n(s) ) \|^{\frac {q_j }{q_j  -1}}_{V_j^*}
   ds \right )
    \le   
 M_2 (1+  \| x \|^{2+ \beta_{1,j}  }_H ),
  $$  
  and
  $$
   \E \left (
  \int_0^T
  \| B(s, Z_n (s))\|^{2}_{\call_2(U,H)} ds
  \right )
    \le   
 M_2 (1+  \| x \|^{2  }_H ).
  $$  
   \end{lem}
  
  \begin{proof}
  By  \eqref{h4a}, Young's inequality  and Lemma \ref{ues1} we get
    for all $j=1,\cdots, J$,
     $$
     \E \left (
     \int_0^T
   \| A_j (s,   Z_n(s) ) \|^{\frac {q_j}{q_j -1}}_{V_j^*}
   ds \right )
    \le   
    \alpha_2
    \E \left (\int_0^T
      \| Z_n (s) \|^{q_j}_{V_j} ds
      \right )
     $$
$$
+ \alpha_2 
 \E \left (
  \sup_{s\in [0,T]}  \| Z_n (s)  \|^{ \beta_{1,j}} _H 
 \int_0^T
      \| Z_n (s) \|^{q_j}_{V_j} ds
    \right )
    $$
    $$
   + \int_0^T |g(s)| ds
   \left ( 1+   \E \left ( 
   \sup_{s\in [0,T]}
    \| Z_n (s)  \|  ^{2+  \beta_{1,j} }  _H 
    \right ) \right )
 $$
 $$
   \le   
    \alpha_2
    \E \left (\int_0^T
      \| Z_n (s) \|^{q_j}_{V_j} ds
      \right )
     +
  \alpha_2  \beta_{1,j}
  (2+  \beta_{1,j} )^{-1}
 \E \left (
  \sup_{s\in [0,T]}  \| Z_n (s)  \|^{2+ \beta_{1,j}} _H 
  \right )
  $$
  $$
  + 2  \alpha_2  
  (2+  \beta_{1,j} )^{-1}
  \E \left (
  \left (
 \int_0^T
      \| Z_n (s) \|^{q_j}_{V_j} ds
      \right )^{1 +{\frac 12}  \beta_{1,j} }
    \right )
    $$
    \be\label{ues2 p1a}
   + \int_0^T |g(s)| ds
   \left ( 1+   \E \left ( 
   \sup_{s\in [0,T]}
    \| Z_n (s)  \|  ^{2+  \beta_{1,j}}  _H 
    \right ) \right ).
    \ee
    By \eqref{h5b} 
    we see that
    $ 1\le 1+ {\frac 12}
     \beta_{1,j}
    < {\frac 12 } +\gamma_{1,j}\gamma_{2,j}^{-1}$
    for all $j=1,\cdots, J$, and hence
    $p= 1+  {\frac 12}
     \beta_{1,j} $ satisfies
    \eqref{main p}. Then by Lemma \ref{ues1}
     we get
     for all $j=1,\cdots, J$,
    \be\label{ues2 p1}
     \E \left (
     \int_0^T
   \| A_j (s,   Z_n(s) ) \|^{\frac {q_j}{q_j -1}}_{V_j^*}
   ds \right )
  \le c_1 (1+ \| x  \|^{2+  \beta_{1,j}   }_H ),
 \ee
 where $c_1=c_1 (g, T, \beta)>0$ is a constant.
 
 In addition, by  \eqref{h5a}
  and Lemma \ref{ues1} we have
  $$
  \E \left (
  \int_0^T
  \| B(s, Z_n (s))\|^{2}_{\call_2(U,H)} ds
  \right )
   \le 
    \E \left (\int_0^T \sum_{j=1}^J\gamma_{2,j}
      \| Z_n (s) \|^{{q}_j}_{V_j} ds 
     \right )
  $$
$$
+ 
  \int_0^T |g (s )|  ds  
  \left (1+  \E
  \left ( \sup_{s\in [0,T]}   \|Z_n (s) \|^{2}_H 
  \right ) \right )
  \le   
  c_2 (1+\|x \|^{2}_H ),
$$
where  $c_2=c_2 (g,T)>0$
is a constant, which along with
  \eqref{ues2 p1}  completes
    the proof.
  \end{proof}

   \begin{lem}\label{ues3}
   If  {\bf (H1)}-{\bf (H5)} are fulfilled,
then   for every $\sigma \in\left(0, \frac{1}{2}\right), T>0$ and $x \in H$, there exists a constant $M_{3}=M_{3}(\sigma, T, x)>0$ such that the solution $Z_{n}$ of 
\eqref{ode1}
satisfies, for all $n \in \mathbb{N}$,
 \begin{equation}
\mathbb{E}\left(\left\|Z_{n}\right\|_{W^{\sigma, 2}\left(0, T ; \mathscr{H}^{*}\right)}^{q}\right) \leq M_{3}, 
\end{equation}
 where $q=\min \left\{2, \frac{\tilde{q}}{\tilde{q}-1}\right\}$ and $\tilde{q}$ is the number given by \eqref{qj}.
 \end{lem}
 
 \begin{proof}
  We will derive the uniform estimates for
   the terms on the right-hand side of 
   \eqref{ode1}. For the first term, for very $v \in \mathscr{H}$, 
   by \eqref{basis0} and 
   \eqref{basis4}  we have
 $$
\begin{gathered}
\left|\left(P_{n} A\left(s, Z_{n}(s)\right), v\right)_{\left(\mathscr{H}^{*}, \mathscr{H}\right)}\right|=\left|\left(A\left(s, Z_{n}(s)\right), P_{n} v\right)_{\left(V^{*}, V\right)}\right| \\
\leq\left\|A\left(s, Z_{n}(s)\right)\right\|_{V^{*}}\left\|P_{n} v\right\|_{V} \leq C_{\mathscr{H}}\left\|A\left(s, Z_{n}(s)\right)\right\|_{V^{*}}\left\|P_{n} v\right\|_{\mathscr{H}} \leq C_{\mathscr{H}}\left\|A\left(s, Z_{n}(s)\right)\right\|_{V^{*}}\|v\|_{\mathscr{H}},
\end{gathered}
$$
 and hence
 $$
\left\|P_{n} A\left(s, Z_{n}(s)\right)\right\|_{\mathscr{H}^{*}} \leq C_{\mathscr{H}}\left\|A\left(s, Z_{n}(s)\right)\right\|_{V^{*}},
$$
 which along with Lemma \ref{ues2} shows that
 \begin{equation}\label{ues3 p1}
\mathbb{E}\left(\int_{0}^{T}\left\|P_{n} A\left(s, Z_{n}(s)\right)\right\|_{\mathscr{H}^{*}}^{\frac{\tilde{q}}{\bar{q}-1}} d s\right) \leq c_{1}, 
\end{equation}
 where $c_{1}=c_{1}(T, x)>0$ is a constant. By 
 \eqref{ues3 p1} we have
 \begin{equation}\label{ues3 p2a}
\mathbb{E}\left(\left\|\int_{0}^{\cdot} P_{n} A\left(s, Z_{n}(s)\right) d s\right\|_{W^{1, \frac{\tilde{q}}{\tilde{q}-1}}
\left(0, T ; \mathscr{H}^{*}\right)}^
{\frac{\tilde{q}}{\tilde{q}-1}}
\right) \leq c_{1} . 
\end{equation}

Since for all  $\sigma
\in (0, {\frac 12})$, 
the space 
${W^{1, \frac{\tilde{q}}{\tilde{q}-1}}
\left(0, T ; \mathscr{H}^{*}\right)}
$ is embedded into 
$
{W^{\sigma,  2}
\left(0, T ; \mathscr{H}^{*}\right)}$,
by   \eqref{ues3 p2a} 
 we get,  for every $\sigma \in\left(0, \frac{1}{2}\right)$,
 \begin{equation}\label{ues3 p2}
\mathbb{E}\left(\left\|\int_{0}^{.} P_{n} A\left(s, Z_{n}(s)\right) d s\right\|_{W^{\sigma, 2}\left(0, T ; \mathscr{H}^{*}\right)}^{\frac{\tilde{q}}{\bar{q}-1}}\right) \leq c_{2}, 
\end{equation}
 where $c_{2}=c_{2}(\sigma, T, x)>0$ is a constant.
 
On the other hand, by 
\cite[Lemma 2.1]{fla1} and Lemma  \ref{ues2}
we infer that for every $\sigma \in\left(0, \frac{1}{2}\right)$, 
there exists $c_{3}=c_{3}(\sigma)>0$ such that
 \begin{equation}\label{ues3 p3}
\mathbb{E}\left(\left\|\int_{0} P_{n} B\left(s, Z_{n}(s)\right) Q_{n} d W(s)\right\|_{W^{\sigma, 2}(0, T ; H)}^{2}\right) \leq c_{3} \mathbb{E}\left(\int_{0}^{T}\left\|B\left(s, Z_{n}(s)\right)\right\|_{\mathcal{L}_{2}(U, H)}^{2}\right) \leq c_{4} ,
\end{equation}
 where $c_{4}=c_{4}(\sigma, T, x)>0$ is a constant.
  Let $q=\min \left\{2, \frac{\tilde{q}}{\tilde{q}-1}\right\}$. It follows from \eqref{ode1}
   and \eqref{ues3 p2}
   -\eqref{ues3 p3}  that
 \begin{equation*}
\mathbb{E}\left(\left\|Z_{n}\right\|_{W^{\sigma, 2}\left(0, T ; \mathscr{H}^{*}\right)}^{q}\right) \leq c_{5}  
\end{equation*}
 where $c_{5}=c_{5}(\sigma, T, x)>0$, which completes the proof.
 \end{proof}

\subsection{Tightness of approximate solutions}
In this section, we prove the tightness
of the sequence
$\{Z_n\}_{n=1}^\infty$  of approximate
solutions in an appropriate space.

Let $L_{w^{*}}^{\infty}(0, T ; H)$ be the space $L^{\infty}(0, T ; H)$ endowed with the weak-* topology and $L_{w}^{q_{j}}\left(0, T ; V_{j}\right)$ be the space $L^{q_{j}}\left(0, T ; V_{j}\right)$ endowed with the weak topology. Denote by

\begin{equation}\label{caly}
\mathcal{Y}=L^{\underline{q}}(0, T ; H) \bigcap L_{w^{*}}^{\infty}(0, T ; H) \bigcap\left(\bigcap_{j=1}^{J} L_{w}^{q_{j}}\left(0, T ; V_{j}\right)\right), 
\end{equation}
 endowed with the supremum topology $\mathcal{T}$ of the corresponding topologies, 
 where $\underline{q}$ is the number given by 
 \eqref{qj}.

We now prove the   tightness of $\left\{Z_{n}\right\}_{n=1}^{\infty}$ in $\mathcal{Y}$, which 
is crucial  to  apply the
Skorokhod-Jakubowski representation theorem
in  the nonmetric space  $\mathcal{Y}$.

\begin{lem} 
\label{tig1}
   Suppose    {\bf (H1)}-{\bf (H5)} hold
   and the embedding $V \subseteq H$ is compact. Then the sequence $\left\{Z_{n}\right\}_{n=1}^{\infty}$ is tight in $(\mathcal{Y}, \mathcal{T})$.
    \end{lem}

      \begin{proof}
    By Lemmas \ref{ues1} and 
    \ref{ues3}, we find that for every $\sigma \in\left(0, \frac{1}{2}\right), T>0$ and $x \in H$, there exists $c_{1}=c_{1}(\sigma, T, x)>0$ such that
 $$
\mathbb{E}\left(\sup _{t \in[0, T]}\left\|Z_{n}(t)\right\|_{H}^{2}+\sum_{j=1}^{J} \int_{0}^{T}\left\|Z_{n}(t)\right\|_{V_{j}}^{q_{j}} d t+\left\|Z_{n}\right\|_{W^{\sigma, 2}\left(0, T ; \mathscr{H}^{*}\right)}^{q}\right) \leq c_{1},
$$
 where $q$ is the number given in Lemma 
 \ref{ues3}. Therefore, for every $\varepsilon>0$, there exists $R=$ $R(\varepsilon, \sigma, T, x)>0$ such that for all $n \in \mathbb{N}$,
 $$
\mathbb{P}\left(\sup _{t \in[0, T]}\left\|Z_{n}(t)\right\|_{H}+\sum_{j=1}^{J} \int_{0}^{T}\left\|Z_{n}(t)\right\|_{V_{j}}^{q_{j}} d t+\left\|Z_{n}\right\|_{W^{\sigma, 2}\left(0, T ; \mathscr{H}^{*}\right)}>R\right)<\varepsilon .
$$
 Let $K$ be the subset of $\mathcal{Y}$ given by
 \begin{equation}\label{tig1 p1}
K=\left\{y \in \mathcal{Y}: \sup _{t \in[0, T]}\|y(t)\|_{H}+\sum_{j=1}^{J} \int_{0}^{T}\|y(t)\|_{V_{j}}^{q_{j}} d t+\|y\|_{W^{\sigma, 2}\left(0, T ; \mathscr{H}^{*}\right)} \leq R\right\} 
\end{equation}
 It is clear that $\mathbb{P}\left(Z_{n} \in K_{1}\right)>1-\varepsilon$ for all $n \in \mathbb{N}$. Next, we prove $K$ is a compact subset of $(\mathcal{Y}, \mathcal{T})$.

Since for every $j=1, \cdots, J, L^{q_{j}}\left(0, T ; V_{j}\right)$ is a separable reflexive Banach space, we find that the weak topology of $L^{q_{j}}\left(0, T ; V_{j}\right)$ on the bounded subset $K$ is metrizable. Since $K$ is bounded in $L^{\infty}(0, T ; H)=\left(L^{1}(0, T ; H)\right)^{*}$ and $L^{1}(0, T ; H)$ is separable, we infer that the weak-* topology of $L^{\infty}(0, T ; H)$ is metrizable on $K$. Since $W^{\sigma, 2}\left(0, T ; \mathscr{H}^{*}\right)$ is also a metric space, we conclude that the topological space $(K, \mathcal{T})$ is metrizable. Consequently, $(K, \mathcal{T})$ is compact if and only if it is sequentially compact.

In the sequel, we 
  prove $(K, \mathcal{T})$ is sequentially compact. Let $\left\{y_{n}\right\}_{n=1}^{\infty}$ be an arbitrary sequence in $K$. It is sufficient to show $\left\{y_{n}\right\}_{n=1}^{\infty}$ has a convergent subsequence in each space involved in
  \eqref{caly}. By \eqref{tig1 p1}
   we see that $\left\{y_{n}\right\}_{n=1}^{\infty}$ is bounded both in $L^{q_{j}}\left(0, T ; V_{j}\right)$ and $L^{\infty}(0, T ; H)$, and hence, up to a subsequence, $\left\{y_{n}\right\}_{n=1}^{\infty}$ is weakly convergent in $L^{q_{j}}\left(0, T ; V_{j}\right)$, and weak-* convergent in $L^{\infty}(0, T ; H)$. It remains to show $\left\{y_{n}\right\}_{n=1}^{\infty}$ has a convergent subsequence in $L^{q}(0, T ; H)$.

By \eqref{tig1 p1}  we see that there exists $c_{2}=c_{2}(R)>0$ such that for all $n \in \mathbb{N}$,
 \begin{equation}\label{tig1 p2}
\int_{0}^{T}\left\|y_{n}(t)\right\|_{V}^{\underline{q}} d t \leq c_{2} .
\end{equation}
 By \eqref{tig1 p1}-\eqref{tig1 p2}
 we find that $\left\{y_{n}\right\}_{n=1}^{\infty}$ is bounded in $L^{\underline{q}}(0, T ; V) \bigcap W^{\sigma, 2}\left(0, T ; \mathscr{H}^{*}\right)$ for every fixed $\sigma \in\left(0, \frac{1}{2}\right)$, and hence by the compactness of the embedding $V \subseteq H$, we infer from 
 \cite[Corollary 5]{sim1}
  that $\left\{y_{n}\right\}_{n=1}^{\infty}$ is precompact in $L^{\underline{q}}(0, T ; H)$. This completes the proof.
   \end{proof}

   By Lemma \ref{tig1}
 and the Skorokhod-Jakubowski representation theorem on a topological space given by Proposition 
 \ref{prop_sj} in Appendix, we have:

\begin{lem} 
\label{tig2}
   Suppose       {\bf (H1)}-{\bf (H5)}  hold,
   $p$ satisfies \eqref{main p}
   and the embedding $V \subseteq H$ is compact. 
   Then there exist a probability space $(\widetilde{\Omega}, \widetilde{\mathcal{F}}, \widetilde{\mathbb{P}})$, random variables $(\widetilde{Z}, \widetilde{W})$ and $\left(\widetilde{Z}_{n}, \widetilde{W}_{n}\right)$ defined on $(\widetilde{\Omega}, \widetilde{\mathcal{F}}, \widetilde{\mathbb{P}})$ such that, up to a subsequence:
   
(i) The law of $\left(\widetilde{Z}_{n}, \widetilde{W}_{n}\right)$ coincides with $\left(Z_{n}, W\right)$ in $\mathcal{Y}
 \times C\left([0, T],
 U_{0}\right)$ 
 for all $n \in \mathbb{N}$.

(ii) $(\widetilde{Z}_{n},
\widetilde{W}_{n})
 \rightarrow 
 (\widetilde{Z}, \widetilde{W})$
  in $\mathcal{Y}
 \times C\left([0, T],
 U_{0}\right)$,
   $\widetilde{\mathbb{P}}$-almost surely.

(iii) $\left\{\widetilde{Z}_{n}\right\}_{n=1}^{\infty}$ satisfies the uniform estimates: for all $n \in \mathbb{N}$,
 \begin{align}
& \widetilde{\mathbb{E}}\left(\sup _{t \in[0, T]}\left\|\widetilde{Z}_{n}(t)\right\|_{H}^{2 p}\right)+\widetilde{\mathbb{E}}\left(\left(\int_{0}^{T} \sum_{j=1}^{J}\left\|\widetilde{Z}_{n}(t)\right\|_{V_{j}}^{q_{j}} d t\right)^{p}\right) 
 \nonumber\\
+ & \widetilde{\mathbb{E}}\left(\int_{0}^{T}\left\|\widetilde{Z}_{n}(s)\right\|_{H}^{2 p-2} \sum_{j=1}^{J}\left\|\widetilde{Z}_{n}(s)
\right\|_{V_{j}}^{q_{j}} d s\right) \leq M_{3}\left(1+\|x\|_{H}^{2 p}\right), \label{tig2 1} 
\end{align}
 and
 \begin{align}
& \widetilde{\mathbb{E}}\left(\sup _{t \in[0, T]}\|\widetilde{Z}(t)\|_{H}^{2 p}\right)+\mathbb{E}\left(\left(\int_{0}^{T} \sum_{j=1}^{J}\|\widetilde{Z}(t)\|_{V_{j}}^{q_{j}} d t\right)^{p}\right) 
\nonumber\\
+ & \widetilde{\mathbb{E}}\left(\int_{0}^{T}\|\widetilde{Z}(s)\|_{H}^{2 p-2} \sum_{j=1}^{J}\|\widetilde{Z}(s)\|_{V_{j}}^{q_{j}} d s\right) \leq M_{3}\left(1+\|x\|_{H}^{2 p}\right) . 
\label{tig2 2} 
\end{align}

In addition, for all $n \in \mathbb{N}$ and $j=1, \cdots, J$,
   \begin{equation}\label{tig2 3}
\widetilde{\mathbb{E}}\left(\int_{0}^{T}\left\|A_{j}\left(s, \widetilde{Z}_{n}(s)\right)\right\|_{V_{j}^{*}}^{\frac{q_{j}}{q_{j}-1}} d s\right)+\widetilde{\mathbb{E}}\left(\int_{0}^{T}\left\|P_{n} B\left(s, \widetilde{Z}_{n}(s)\right) Q_{n}\right\|_{\mathcal{L}_{2}(U, H)}^{2} d s\right)
\leq M_{3}\left(1+\|x\|_{H}^{2+\beta_{1,j}}\right), 
\end{equation}
 where $M_{3}=M_{3}(T, p)>0$ is a constant.
     \end{lem}  
   
   \begin{proof}
We will apply Proposition 
\ref{prop_sj} to the sequence $\left\{\left(Z_{n}, W\right)\right\}_{n=1}^{\infty}$ in the topological space $\mathcal{Y} \times$ $C\left([0, T], U_{0}\right)$. To that end, we first need to find a sequence of continuous real-valued functions on $(\mathcal{Y}, \mathcal{T})$ that separates points of $\mathcal{Y}$.

Recall that $\mathcal{T}$ is the supremum topology of the corresponding topologies in $\mathcal{Y}$. Consequently, by 
\eqref{caly} 
 we see that $\mathcal{T}$ is stronger than each of the following topology:
 $$
L^{\underline{q}}(0, T ; H), 
\ \ L_{w^{*}}^{\infty}(0, T ; H) \quad \text { and } \quad L_{w}^{q_{j}}\left(0, T ; V_{j}\right), \quad j=1, \cdots, J .
$$
    Since $\left(L^{\underline{q}}(0, T ; V)\right)^{*}$ is separable, there exists a sequence $\{\phi_{j}^{*}\}_{j=1}^{\infty}$ in $\left(L^{\underline{q}}(0, T ; V)\right)^{*}$ that separates points of $L^{\underline{q}}(0, T ; V)$. Since $\mathcal{T}$ is stronger than the topology of $L^{\underline{q}}(0, T ; V)$, we find that $\phi_{j}^{*}$ : $(\mathcal{Y}, \mathcal{T}) \rightarrow \mathbb{R}$ is continuous and $\{\phi_{j}^{*}\}_{j=1}^{\infty}$ separates points of $\mathcal{Y}$.

Note that $C\left([0, T], U_{0}\right)$ is a polish space, which along with Lemma \ref{tig1}
shows that all conditions of Proposition 
\ref{prop_sj} are satisfied. Then (i)-(ii) in Lemma 
\ref{tig2}
 follows from Proposition 
 \ref{prop_sj}
  immediately. By (i) and Lemma \ref{ues1}, we obtain 
  \eqref{tig2 1}. 
  By (ii) and \eqref{tig2 1},  we obtain 
  \eqref{tig2 2}. The uniform estimate 
  \eqref{tig2 3}
   follows from \eqref{tig2 1}
   and the idea of Lemma \ref{ues2}. 
   \end{proof}

Note that Lemma \ref{tig2}
(i)  implies that
$\widetilde{\p}
(\widetilde{Z}_{n}
\in C([0,T], H_n) )
= {\p}
( {Z}_{n}
\in C([0,T], H_n) )$.
Since
${\p}
( {Z}_{n}
\in C([0,T], H_n) ) =1$, we infer that
$\widetilde{Z}_{n}
\in C([0,T], H_n)$, 
$\widetilde{\p}$-almost surely.

    Given $n\in \N$, let
    $(\widetilde{\calf}^n_t)_{t\in [0,T]}$ be the filtration
    which satisfies   the usual condition   generated by
    $$
    \{  \widetilde{Z}_n (s), \ 
    \widetilde{W}_n (s): \ s\in [0, t]\}.
    $$
    Similarly,
     let
    $(\widetilde{\calf}_t)_{t\in [0,T]}$ be the filtration
    which satisfies   the usual condition   generated by
    $$
    \{  \widetilde{Z} (s), \ 
    \widetilde{W} (s): \ s\in [0, t]\}.
    $$
    Then $   \widetilde{W}^n $
    is an  $(\widetilde{\calf}^n_t)$-cylindrical Wiener process, 
    and
     $   \widetilde{W} $
    is an  $(\widetilde{\calf}_t)$-cylindrical Wiener process 
    on  $U$, respectively.
    
    Since $Z_n$ satisfies \eqref{ode1}, by Lemma
    \ref{tig2}  (i) we find  that
    $ \widetilde{Z}_n$ satisfies
   \be\label{ode1n}
     \widetilde{Z}_n (t)
     = P_n x
     +\int_0^t P_n A(s,
      \widetilde{Z}_n (s)) ds
      +
      \int_0^t P_n 
      B(s,  \widetilde{Z}_n (s)) Q_n d
       \widetilde{W}_n (s),
       \ee
    which can be proved by the argument of    \cite{ben1}.
    Next, we prove $
    (\widetilde{\Omega},
    \widetilde{\calf}, ( \widetilde{\calf}_t)_{t\in [0,T]},
    \widetilde{\p}, \widetilde{Z} , \widetilde{W})
    $ is a martingale solution of \eqref{sde1}-\eqref{sde2}.

\subsection{Proof of Theorem \ref{main}}

We now prove Theorem 
\ref{main}
 by the monotone argument.

    {\bf Step (i)}: Weak convergence.
     By Lemma \ref{tig2}, we infer that
    there exist 
    $\widetilde{B} \in 
      L^{2}
    ([0,T] \times \widetilde{\Omega}, 
    \call_2(U,H))$ and 
    $ \widetilde{A}_j \in
     \ L^{{\frac {q_j}{q_j -1}}}
    ([0,T] \times \widetilde{\Omega}, V^*_j)$
    for every $ j=1,\cdots, J$,   such that,
    up to a subsequence,
     \be\label{ms1 p1a}
    \widetilde{Z}_n
    \to \widetilde{Z}
    \ \text{ weakly   in   } \ L^2
    ([0,T] \times \widetilde{\Omega}, H),  
      \ee
    \be\label{ms1 p1}
    \widetilde{Z}_n
    \to \widetilde{Z}
    \ \text{ weakly   in   } \ L^{q_j}
    ([0,T] \times \widetilde{\Omega}, V_j), \quad \forall \  j=1,\cdots, J,
      \ee
       \be\label{ms1 p2}
   A_j (\cdot,  \widetilde{Z}_n)
    \to \widetilde{A}_j
    \ \text{ weakly   in   } \ L^{{\frac {q_j}{q_j -1}}}
    ([0,T] \times \widetilde{\Omega}, V^*_j), \quad \forall \  j=1,\cdots, J.
      \ee
        \be\label{ms1 p3}
        P_n B(
   \cdot,  \widetilde{Z}_n) Q_n 
    \to \widetilde{B}
    \ \text{ weakly   in   } \ 
      L^{2}
    ([0,T] \times \widetilde{\Omega}, 
    \call_2(U,H)).
      \ee
     On the other hand, by Lemma \ref{tig2} (ii)
     and \eqref{caly} we have,
     $ \widetilde{\p}$-almost surely,
       \be\label{ms1 p4}
  \lim_{n\to \infty}
   \| \widetilde{Z}_n
 - \widetilde{Z}\|_{ L^{\underline{q}}(0,T; H)}
 = 0,
      \ee
            \be\label{ms1 p4a}
    \widetilde{Z}_n
    \to \widetilde{Z}
    \ \text{ weakly  \   in   } \ 
         L^{ {q_j}}(0,T; V_j),
         \quad \forall \  j=1, \cdots, J,
         \ee 
        and
                \be\label{ms1 p5}
    \widetilde{Z}_n
    \to \widetilde{Z}
    \ \text{ weak-*  \   in   } \ 
          L^\infty
    (0,T; H).
         \ee
By
\eqref{ms1 p4a} and
 \eqref{ms1 p5} we infer that  $ \widetilde{\p}$-almost surely,
   \be\label{ms1 p6}
  \sup_{1\le j\le J } 
  \   \sup_{n\in \N}
     \|   \widetilde{Z}_n
      \|_{  L^{ {q_j}}(0,T; V_j)}
      <\infty
      \quad \text{and}
      \quad
  \sup_{n\in \N}
   \|  \widetilde{Z}_n\|_{  L^\infty
    (0,T; H)} <\infty .
     \ee
     Recall that $\{h_k\}_{k=1}^\infty$ is an orthonormal
     basis of $H$ and $h_k \in V$ for all $k\in \N$.
     
       {\bf Step (ii)}:   Prove    $\widetilde{B}
       =B(\cdot, \widetilde{Z})$ almost everywhere
       on $[0,T]
       \times \widetilde{\Omega}$, and  for
       all $t\in [0,T]$, $\widetilde{\p}$-almost surely,
         \be\label{limeq}
     \widetilde{Z}  (t)
     =   x
     +\int_0^t \sum_{j=1}^J  \widetilde{A} _j (s) ds
      +
      \int_0^t 
      \widetilde{B} (s )   d
       \widetilde{W}  (s) \quad \text{in } \ V^*,
       \ee
       $\widetilde{\p}$-almost surely.

      We first prove    $\widetilde{B}
       =B(\cdot, \widetilde{Z})$ almost everywhere
       on $[0,T]
       \times \widetilde{\Omega}$.
       To that end,   we identify
         $h_k^* \in H^*$ with $h_k\in H$
         by Riesz's representation theorem.
       Then  for all $n\ge k$, we have
   $$ \| h_k^* P_n B(
   \cdot,  \widetilde{Z}_n)Q_n
   - 
    h_k^* B(
   \cdot,  \widetilde{Z})
   \|^2_{
   L^2([0,T]\times  \widetilde{\Omega},
   \call_2(U,\R)  )
   }
   $$
    $$ 
   \le  3
    \| h_k^* P_n  \left ( B(
   \cdot,  \widetilde{Z}_n) 
   -   B(
   \cdot,  \widetilde{Z})
   \right ) Q_n
   \|^2_{
   L^2([0,T]\times  \widetilde{\Omega},
   \call_2(U,\R)  )
   }
   $$
    $$ 
   + 3
    \| h_k^* P_n B(
   \cdot,  \widetilde{Z})
   \left ( Q_n
   -  I \right ) 
   \|^2_{
   L^2([0,T]\times  \widetilde{\Omega},
   \call_2(U,\R)  )
   }
   +   3
    \| h_k^* \left ( P_n -I \right )
     B(
   \cdot,  \widetilde{Z})
   \|^2_{
   L^2([0,T]\times  \widetilde{\Omega},
   \call_2(U,\R)  )
   }
   $$
    $$ 
   =3
    \|   h_k^* (B(
   \cdot,  \widetilde{Z}_n) 
   -   B(
   \cdot,  \widetilde{Z}))Q_n
   \|^2_{
   L^2([0,T]\times  \widetilde{\Omega},
   \call_2(U,\R)  )
   }
   $$
$$
   + 3
    \| h_k^* B(
   \cdot,  \widetilde{Z})
   \left ( Q_n
   -  I \right ) 
   \|^2_{
   L^2([0,T]\times  \widetilde{\Omega},
   \call_2(U,\R)  )
   }
$$
$$ 
\le 3
    \|   h_k^* (B(
   \cdot,  \widetilde{Z}_n) 
   -   B(
   \cdot,  \widetilde{Z})) 
   \|^2_{
   L^2([0,T]\times  \widetilde{\Omega},
   \call_2(U,\R)  )
   }
   $$
\be\label{ms1 p7a}
   + 3
    \| h_k^* B(
   \cdot,  \widetilde{Z})
   \left ( Q_n
   -  I \right ) 
   \|^2_{
   L^2([0,T]\times  \widetilde{\Omega},
   \call_2(U,\R)  )
   }.
   \ee
   
      By \eqref{ms1 p4} and \eqref{ms1 p6},
       we obtain from \eqref{h5c} that,
       for every fixed $k\in \N$,
       $\widetilde{\p}$-almost
         surely,   
                  \be\label{ms1 p7b} 
         \lim_{n\to \infty}  
          \|   h_k^* (B(
   \cdot,  \widetilde{Z}_n) 
   -   B(
   \cdot,  \widetilde{Z})) 
   \|^2_{
   L^2([0,T]\times  \widetilde{\Omega},
   \call_2(U,\R)  )
   }=0.
   \ee 
   On the other hand, since 
 $h^*_k B(
   \cdot,  \widetilde{Z})
 \in 
   L^2([0,T]\times  \widetilde{\Omega},
   \call_2(U,\R)  )$, we find that
   the last   term  on the
   right-hand side of \eqref{ms1 p7a}
   converge to zero, and thus
   by \eqref{ms1 p7a}-\eqref{ms1 p7b} we obtain,
   for every $k\in \N$,
\be\label{ms1 p7c}
\lim_{n\to \infty}
 h^*_k P_n B(
   \cdot,  \widetilde{Z}_n)Q_n
 = 
    h^*_k B(\cdot,  \widetilde{Z})
 \quad \text{strongly in }\ 
   L^2([0,T]\times  \widetilde{\Omega},
   \call_2(U,\R )).
  \ee
  
 By \eqref{ms1 p3} we have, 
  for every $k\in \N$,
\be\label{ms1 p7d}
\lim_{n\to \infty}
 h^*_k P_n B(
   \cdot,  \widetilde{Z}_n)Q_n
 = 
    h^*_k  \widetilde{B}
 \quad \text{weakly in }\ 
   L^2([0,T]\times  \widetilde{\Omega},
   \call_2(U,\R )).
  \ee
  By \eqref{ms1 p7c}-\eqref{ms1 p7d} we get,
   for every $k\in \N$,
$$
   h^*_k  \widetilde{B}
    =  h^*_k B(\cdot,  \widetilde{Z})
    \quad \text{  in }\ 
   L^2([0,T]\times  \widetilde{\Omega},
   \call_2(U,\R )),
$$
and hence  $\widetilde{B}
    =     B(\cdot,  \widetilde{Z})$
    in $\call_2(U,H )$
    a.e. on $[0,T] \times \widetilde{\Omega}$.

   We now prove \eqref{limeq}.
         Let $\xi$ be an arbitrary
       function in $L^\infty([0,T]\times
       \widetilde{\Omega})$.
       Then all $n, k\in \N$ with $n \ge k$,
       by \eqref{ode1n} we have
       $$
       \widetilde{\E} \left (
       \int_0^T \xi (t)
       (
     \widetilde{Z}_n (t), h_k)_H dt
     \right )
     = \widetilde{\E} \left (
       \int_0^T \xi (t)
       (  P_n x, h_k)_H dt
       \right )
       $$
       $$
     +\sum_{j=1}^J 
      \widetilde{\E} \left (
       \int_0^T \xi (t)
      \left  (
     \int_0^t  ( P_n A_j (s,
      \widetilde{Z}_n (s)), h_k)_{(V^*,V)} ds
     \right  )dt \right )
     $$
  \be\label{ms1 p8a}
      + \widetilde{\E} \left (
       \int_0^T \xi (t)
      \left  (
      \int_0^t P_n 
      B(s,  \widetilde{Z}_n (s)) Q_n d
       \widetilde{W}_n (s), h_k\right )_H dt
       \right ).
     \ee
     We first deal with the convergence of the
    last term in \eqref{ms1 p8a}, for which  we have 
  for all $n\ge k$,    
    $$   \widetilde{\E} \left (
       \int_0^T \xi (t)
      \left  (
      \int_0^t P_n 
      B(s,  \widetilde{Z}_n (s)) Q_n d
       \widetilde{W}_n (s), h_k\right )_H dt
       \right )
       $$
          \be\label{ms1 p8b} 
          =
       \widetilde{\E} \left (
       \int_0^T \xi (t) \left (
   \int_0^t  
      h^*_k P_n  B(s,  \widetilde{Z}_n (s))Q_n   d
       \widetilde{W}_n (s) \right ) dt
       \right ).
       \ee
          By    
        Lemma \ref{tig2} (ii) and \eqref{ms1 p7c},
  we infer  from
    \cite[Lemma 2.1]{deb1} that
                $$
    \lim_{n\to \infty}
       \int_0^T
        \left |
         \int_0^t
       h^*_k P_n  B(s,  \widetilde{Z}_n (s))Q_n   d
       \widetilde{W}_n (s)
    -
            \int_0^t
      h^*_k   B(s,  \widetilde{Z} (s))   d
       \widetilde{W} (s)  
       \right |^2   dt =0 ,
$$
         in probability, and hence
                  \be\label{ms1 p8c} 
    \lim_{n\to \infty}
       \int_0^T
         \left |
         \int_0^t
       h^*_k P_n  B(s,  \widetilde{Z}_n (s))Q_n   d
       \widetilde{W}_n (s)
    -
            \int_0^t
      h^*_k   B(s,  \widetilde{Z} (s))   d
       \widetilde{W} (s)  
       \right |    dt      =0  \ 
       \ \text{in probability}.
       \ee  
      By \eqref{h5a} and \eqref{tig2 1} we get
        $$
         \widetilde{\E} \left (
  \int_0^T \|
      B(s,  \widetilde{Z}_n (s)) 
     \|^2_{\call_2(U, H)} ds
    \right )
    $$
    $$
    \le
    \sum_{j=1}^J \gamma_{2,j}
    \widetilde{\E} \left (
  \int_0^T \|  \widetilde{Z}_n (s) \|^{q_j}_{V_j}
  ds
  \right )
  + 
  \left (1 +  \widetilde{\E} \left (
  \sup_{0\le s\le T} \|  \widetilde{Z}_n (s) \|^2_H
  \right ) \right )
  \int_0^T |g(s)| ds
  $$
      \be\label{ms1 p8d}
  \le c_1 (1+ \| x\|^2_H),
\ee
  where $c_1 $ is a  positive constant
  depending only on $T$, $g$ and
   $  \gamma_{2,j}$ with $j=1\cdots, J$.
   Similarly, by  \eqref{h5a} and \eqref{tig2 2} we get
     \be\label{ms1 p8e}
           \widetilde{\E} \left (
  \int_0^T \|
        B(s,  \widetilde{Z} (s)) 
     \|^2_{\call_2(U, H)} ds
    \right ) 
  \le c_2 (1+ \| x\|^2_H),
\ee
  where $c_2 $ is a  positive constant
  depending only on $T$, $g$ and
   $  \gamma_{2,j}$ with $j=1\cdots, J$.
        By    \eqref{ms1 p8d}-\eqref{ms1 p8e} we have
      for all $n \ge k$,
             $$ 
     \widetilde{\E} \left (
     \left (
       \int_0^T   \left |   
   \int_0^t  
      h^*_k P_n B(s,  \widetilde{Z} _n(s))  Q_n  d
       \widetilde{W}_n (s) 
       -  \int_0^t  
      h^*_k B(s,  \widetilde{Z} (s))   d
       \widetilde{W}  (s) 
       \right | dt \right )^2
       \right ) 
     $$   
      $$ 
      \le 2 
     \widetilde{\E} \left (
     \left (
       \int_0^T   \left |   
   \int_0^t  
       h^*_k P_n B(s,  \widetilde{Z} _n(s))  Q_n  d
       \widetilde{W}_n (s) 
      \right |  dt
       \right ) ^2\right )
       +
       2 
     \widetilde{\E} \left (
     \left (
       \int_0^T   \left |   
   \int_0^t  
      h^*_k B(s,  \widetilde{Z} (s))   d
       \widetilde{W} (s) \right |  dt
       \right ) ^2\right )
       $$   
         $$ 
      \le 2 T
     \widetilde{\E} \left (
     \int_0^T   \left |   
   \int_0^t  
        h^*_k P_n B(s,  \widetilde{Z} _n(s))  Q_n  d
       \widetilde{W}_n (s) 
       \right |^2  dt
      \right )
       +
       2 T
     \widetilde{\E} \left (
   \int_0^T   \left |   
   \int_0^t  
      h^*_k B(s,  \widetilde{Z} (s))   d
       \widetilde{W} (s) \right |^2  dt
     \right )
       $$   
          $$ 
   = 2 T \int_0^T
     \widetilde{\E} \left (
       \left |   
   \int_0^t  
        h^*_k P_n B(s,  \widetilde{Z} _n(s))  Q_n  d
       \widetilde{W}_n (s) 
       \right |^2  
      \right )dt
       +
       2 T \int_0^T
     \widetilde{\E} \left (
     \left |   
   \int_0^t  
      h^*_k B(s,  \widetilde{Z} (s))   d
       \widetilde{W} (s) \right |^2  
     \right )dt
       $$   
          $$ 
   = 2 T \int_0^T
     \widetilde{\E} \left (
    \int_0^t  \|
       h^*_k P_n B(s,  \widetilde{Z} _n(s))  Q_n  
        \|^2_{\call_2(U,\R)}
     ds  \right )dt
       +
       2 T \int_0^T
     \widetilde{\E} \left (
   \int_0^t   \|
      h^*_k B(s,  \widetilde{Z} (s)) \|^2_{\call_2(U, \R)}
     ds\right )dt
       $$   
      $$
   \le 2  T^2  
     \widetilde{\E} \left (
    \int_0^T  \|
       B(s,  \widetilde{Z}_n  (s))  \|^2_{\call_2(U,H)}
     ds  \right ) 
     + 2  T^2  
     \widetilde{\E} \left (
    \int_0^T  \|
       B(s,  \widetilde{Z}   (s))  \|^2_{\call_2(U,H)}
     ds  \right ) 
     $$
     $$
    \le    T^2  (c_1+ c_2)  (1+ \|x\|^2_H),
    $$
    and hence  the sequence
       $  \{
       \int_0^T
        \left |
         \int_0^t
     h^*_k P_n B(s,  \widetilde{Z}_n (s))  Q_n d
       \widetilde{W}_n (s)  
    -
            \int_0^t
      h^*_k B(s,  \widetilde{Z} (s))   d
       \widetilde{W} (s)  
       \right |    dt
        \}_{n=1}^\infty$
       is uniformly integrable on $\widetilde{\Omega}$.
       Then by \eqref{ms1 p8c}
       and the Vitali's theorem we obtain
          \be\label{ms1 p8f} 
    \lim_{n\to \infty} 
      \widetilde{\E} \left (   \int_0^T
        \left |
         \int_0^t
      h^*_k P_n B(s,  \widetilde{Z}_n (s)) Q_n  d
       \widetilde{W}_n (s)  
    -
            \int_0^t
     h^*_k B(s,  \widetilde{Z} (s))   d
       \widetilde{W} (s)  
       \right |    dt \right )  =0 .
       \ee 
       Note that
      $$
    \left |   \widetilde{\E} \left (
       \int_0^T \xi (t) \left (
   \int_0^t  
      h^*_k P_n B(s,  \widetilde{Z}_n (s)) Q_n  d
       \widetilde{W}_n (s)  
       -  \int_0^t  
      h^*_k B(s,  \widetilde{Z} (s))   d
       \widetilde{W}  (s) 
        \right ) dt
       \right ) 
       \right |
    $$
           $$
           \le \| \xi\|_{L^\infty([0,T]
           \times \widetilde{\Omega})}
     \widetilde{\E} \left (
       \int_0^T   \left |   
   \int_0^t  
       h^*_k P_n B(s,  \widetilde{Z}_n (s)) Q_n  d
       \widetilde{W}_n (s)  
       -  \int_0^t  
      h^*_k B(s,  \widetilde{Z} (s))   d
       \widetilde{W}  (s) 
       \right | dt
       \right ) ,
     $$   
    which together  with 
    \eqref{ms1 p8f} implies that
    \be\label{ms1 p8f1}
   \lim_{n\to \infty}
     \widetilde{\E} \left (
       \int_0^T \xi (t) \left (
   \int_0^t  
      h^*_k P_n B(s,  \widetilde{Z}_n (s)) Q_n  d
       \widetilde{W}_n (s)  
       -  \int_0^t  
      h^*_k B(s,  \widetilde{Z} (s))   d
       \widetilde{W}  (s) 
        \right ) dt
       \right ) 
     =0.
       \ee
    
    It follows from \eqref{ms1 p8b}
   and \eqref{ms1 p8f1}
    that for every $k\in \N$,
       $$
   \lim_{n\to \infty} 
    \widetilde{\E} \left (
       \int_0^T \xi (t)
      \left  (
      \int_0^t P_n 
      B(s,  \widetilde{Z}_n (s)) Q_n d
       \widetilde{W}_n (s), h_k\right )_H dt
       \right )
       $$
        \be\label{ms1 p8f2}
     =  \widetilde{\E} \left (
       \int_0^T \xi (t) \left (
   \int_0^t  
        B(s,  \widetilde{Z} (s))   d
       \widetilde{W} (s), h_k 
        \right )_H dt
       \right ).
     \ee
       Letting $n\to \infty$ in \eqref{ms1 p8a},
       by \eqref{ms1 p1a}-\eqref{ms1 p2} 
       and \eqref{ms1 p8f2}, we obtain 
    for every $k\in \N$,  
       $$
       \widetilde{\E} \left (
       \int_0^T \xi (t)
       (
     \widetilde{Z} (t), h_k)_H dt
     \right )
     = \widetilde{\E} \left (
       \int_0^T \xi (t)
       (   x, h_k)_H dt
       \right )
       $$
       $$
     +
      \widetilde{\E} \left (
       \int_0^T \xi (t)
      \left  (
     \int_0^t \sum_{j=1}^J  (  \widetilde{A}_j
     (s)  , h_k)_{(V^*,V)} ds
     \right  )dt \right )
     $$
  \be\label{ms1 p9}
      + \widetilde{\E} \left (
       \int_0^T \xi (t)
      \left  (
      \int_0^t  
      B(s,  \widetilde{Z} (s))  d
       \widetilde{W} (s), h_k\right )_H dt
       \right ).
     \ee
        Since  $\xi$  is an arbitrary
       function in $L^\infty([0,T]\times
       \widetilde{\Omega})$,
       by \eqref{ms1 p9} we get,
     for almost all $(t,\omega) \in [0,T] \times   \widetilde{\Omega} $,
        \be\label{ms1 p9a}
     \widetilde{Z}  (t)
     =   x
     +\int_0^t \sum_{j=1}^J  \widetilde{A} _j (s) ds
      +
      \int_0^t 
      B(s,  \widetilde{Z}  (s))   d
       \widetilde{W}  (s) \quad \text{in } \ V^*.
       \ee

        By \eqref{tig2 2} and the argument of \cite{kry1}
       we infer that
       there exists an $H$-valued continuous
       $(\widetilde{\calf}_t)$-adapted process
       $\hat{Z}$ such that
       $\hat{Z} =   \widetilde{Z}$ almost everywhere
       on $[0,T] \times   \widetilde{\Omega}$ and
       $\hat{Z} (t)$ is equal to the right-hand side of
       \eqref{ms1 p9a} for all $t\in [0,T]$,
    $ \widetilde{\p}$-almost surely.
   From now on, we identify $\hat{Z}$ with
  $ \widetilde{Z}$, and thus \eqref{ms1 p9a}
  holds   for all $t\in [0,T]$,
    $ \widetilde{\p}$-almost surely.
    This proves \eqref{limeq}.

{\bf 
Step (iii): }  Prove $\sum_{j=1}^{J} \widetilde{A}_{j}=\sum_{j=1}^{J} A_{j}(\cdot, \widetilde{Z})$  
 a.e. on $[0, T] \times \widetilde{\Omega}$.

It follows from \eqref{ms1 p4}   that
 $$
\lim _{n \rightarrow 
\infty}\left\|\widetilde{Z}_{n}
-\widetilde{Z}\right\|_{L^{1}(0, T ; H)}=0,
$$
 which along with 
 \eqref{tig2 1}-\eqref{tig2 2}
  and the Vitali theorem implies that
$$
\lim _{n \rightarrow \infty}\left\|\widetilde{Z}_{n}-\widetilde{Z}\right\|_{L^{1}([0, T] \times \widetilde{\Omega} ; H)}=0,
$$
 and hence, up to a subsequence,
 \begin{equation}\label{ms2 p7}
\lim _{n \rightarrow \infty}\left\|\widetilde{Z}_{n}(t, \omega)-\widetilde{Z}(t, \omega)\right\|_{H}=0, \quad \text { for almost all }(t, \omega) \in[0, T] \times \widetilde{\Omega} .
\end{equation}

Let $\widetilde{X}$ be an arbitrary $H$-valued continuous $(\widetilde{\mathcal{F}}_{t})$-adapted process such that
 \begin{equation}\label{ms2 p10}
\widetilde{\mathbb{E}}\left(\sup _{t \in[0, T]}\|\widetilde{X}(t)\|_{H}^{2}+\sum_{j=1}^{J} \int_{0}^{T}\|\widetilde{X}(t)\|_{V_{j}}^{q_{j}} d t\right)<\infty .  
\end{equation}
 Given $R>0$, let $\tau_{R}(\tilde{X})$ be a stopping time given by
 $$
\tau_{R}(\widetilde{X})=\inf \left\{t \geq 0:\|\widetilde{X}(t)\|_{H}+\int_{0}^{T} \sum_{j=1}^{J}\|\widetilde{X}(t)\|_{V_{j}}^{q_{j}} d t>R\right\} \wedge T,
$$
 with  convention $\inf \emptyset=+\infty$. For simplicity, we often write $\tau_{R}(\widetilde{X})$ as $\tau_{R}$ if no confusion arises.

By Itô's formula and \eqref{ode1n}  we have
$$
\widetilde{E}\left(\left\|\widetilde{Z}_{n}\left(t \wedge \tau_{R}\right)\right\|^{2}\right)-\left\|P_{n} x\right\|_{H}^{2}
$$
$$
\leq \widetilde{E}\left(2 \int_{0}^{t \wedge \tau_{R}}\left(A\left(s, \widetilde{Z}_{n}(s)\right)-A(s, \widetilde{X}(s)), \widetilde{Z}_{n}(s)-\widetilde{X}(s)\right)_{\left(V^{*}, V\right)} d s\right) 
$$
$$
+\widetilde{E}\left(\int_{0}^{t \wedge \tau_{R}}\left\|B\left(s, \widetilde{Z}_{n}(s)\right)-B(s, \widetilde{X}(s))\right\|_{\mathcal{L}_{2}(U, H)}^{2} d s\right)
$$
$$
  +\widetilde{E}\left(2 \sum_{j=1}^{J} \int_{0}^{t \wedge \tau_{R}}\left(\left(A_{j}\left(s, \widetilde{Z}_{n}(s)\right), \widetilde{X}(s)\right)_{\left(V_{j}^{*}, V_{j}\right)}+\left(A_{j}(s, \widetilde{X}(s)), \widetilde{Z_{n}}(s)-\widetilde{X}(s)\right)_{\left(V_{j}^{*}, V_{j}\right)}\right) d s\right) 
  $$
\be\label{ms2 p10a}
 +\widetilde{E}\left(\int_{0}^{t \wedge \tau_{R}}\left(2\left(B\left(s, \widetilde{Z}_{n}(s)\right), B(s, \widetilde{X}(s))\right)_{\mathcal{L}_{2}(U, H)}-\|B(s, \widetilde{X}(s))\|_{\mathcal{L}_{2}(U, H)}^{2}\right) d s\right) 
\ee

Let $\xi \in L^{\infty}([0, T] \times \widetilde{\Omega})$ be a nonnegative function. Then by 
\eqref{h2a}
and
\eqref{ms1 p1}-\eqref{ms1 p3}
, we obtain from  \eqref{ms2 p10a} that
$$
\liminf _{n \rightarrow \infty} \widetilde{E}\left(\int_{0}^{T} \xi(t)\left(\left\|\widetilde{Z}_{n}\left(t \wedge \tau_{R}\right)\right\|^{2}-\left\|P_{n} x\right\|_{H}^{2}\right) d t\right) 
$$
$$
\leq \liminf _{n \rightarrow \infty} \widetilde{E}\left(\int_{0}^{T} \xi(t)\left(\int_{0}^{t \wedge \tau_{R}}\left(g(s)+\varphi(\widetilde{X}(s))+\psi\left(\widetilde{Z}_{n}(s)\right)\right)\left\|\widetilde{Z}_{n}(s)-\widetilde{X}(s)\right\|_{H}^{2} d s\right) d t\right)
$$
$$
+\widetilde{E}\left(
2 \int_{0}^{T}(\xi(t) \sum_{j=1}^{J} \int_{0}^{t \wedge \tau_{R}}((\widetilde{A}_{j}(s), \widetilde{X}(s))_{(V_{j}^{*}, V_{j})}+(A_{j}(s, \widetilde{X}(s)), \widetilde{Z}(s)-\widetilde{X}(s))_{(V_{j}^{*}, V_{j})}) d s) d t\right) 
$$
\be\label{ms2 p10b}
+\widetilde{E}\left(\int_{0}^{T}\left(\xi(t) \int_{0}^{t \wedge \tau_{R}}\left(2(\widetilde{B}(s), B(s, \widetilde{X}(s)))_{\mathcal{L}_{2}(U, H)}-\|B(s, \widetilde{X}(s))\|_{\mathcal{L}_{2}(U, H)}^{2}\right) d s\right) d t\right) .  
\ee

By Itô's formula and \eqref{limeq}
 we have
$$
\widetilde{E}\left(\int_{0}^{T} \xi(t)\left(\left\|\widetilde{Z}\left(t \wedge \tau_{R}\right)\right\|^{2}-\|x\|_{H}^{2}\right) d t\right)
$$
\be \label{ms2 p10c}
=\widetilde{E}\left(\int_{0}^{T}\left(\xi(t) \int_{0}^{t \wedge \tau_{R}}\left(2 \sum_{j=1}^{J}\left(\widetilde{A}_{j}(s), \widetilde{Z}(s)\right)_{\left(V_{j}^{*}, V_{j}\right)}+\|\widetilde{B}(s)\|_{\mathcal{L}_{2}(U, H)}^{2}\right) d s\right) d t\right) . 
\ee

Next, we deal with the limit of  \eqref{ms2 p10b} as $R \rightarrow \infty$. Note that \eqref{ms2 p10} implies that
 \begin{equation} \label{ms2 p11}
\lim _{R \rightarrow \infty} \widetilde{\mathbb{P}}\left(\tau_{R}<T\right)=0.
\ee
 On the other hand, for every $\varepsilon>0$, we have
 $$
\begin{gathered}
dt  \times  
\widetilde{\mathbb{P}} \left((t, \omega) \in[0, T] \times \widetilde{\Omega}:\left\|\widetilde{Z}_{n}\left(t \wedge \tau_{R}\right)-\widetilde{Z}\left(t \wedge \tau_{R}\right)\right\|_{H}>\varepsilon\right) \\
\leq
 dt \times  
 \widetilde{\mathbb{P}}  \left((t, \omega) \in[0, T] \times \widetilde{\Omega}:\left\|\widetilde{Z}_{n}(t)-\widetilde{Z}(t)\right\|_{H}>\varepsilon\right)+T \widetilde{\mathbb{P}}\left(\tau_{R}<T\right),
\end{gathered}
$$
 which along with \eqref{ms2 p7}
  and \eqref{ms2 p11} implies that
 \begin{equation}\label{ms2 p12}
\lim _{n, R \rightarrow \infty}\left\|\widetilde{Z}_{n}\left(t \wedge \tau_{R}\right)-\widetilde{Z}\left(t \wedge \tau_{R}\right)\right\|_{H}=0 \quad \text { in measure in }[0, T] \times \widetilde{\Omega} .
\ee

       By  \eqref{tig2 1}-\eqref{tig2 2}  and  \eqref{ms2 p12}
  we obtain
 $$
\lim _{n, R \rightarrow \infty} \widetilde{E}\left(\int_{0}^{T}\left\|\widetilde{Z}_{n}\left(t \wedge \tau_{R}\right)-\widetilde{Z}\left(t \wedge \tau_{R}\right)\right\|_{H}^{2} d t\right)=0,
$$
 and hence for any nonnegative $\xi \in L^{\infty}([0, T] \times \widetilde{\Omega})$,
 $$
\lim _{n, R \rightarrow \infty} \widetilde{E}\left(\int_{0}^{T} \xi (t)\left(\left(\left\|\widetilde{Z}_{n}\left(t \wedge \tau_{R}\right)\right\|^{2}-\left\|P_{n} x\right\|_{H}^{2}\right)-\left(\left\|\widetilde{Z}\left(t \wedge \tau_{R}\right)\right\|^{2}-\|x\|_{H}^{2}\right)\right) d t\right)=0,
$$
 which implies that
$$
\lim _{R \rightarrow \infty} \liminf _{n \rightarrow \infty} \widetilde{E}\left(\int_{0}^{T} \xi(t)\left(\left\|\widetilde{Z}_{n}\left(t \wedge \tau_{R}\right)\right\|^{2}-\left\|P_{n} x\right\|_{H}^{2}\right) d t\right) 
$$
\be\label{ms2 p13}
=\lim _{R \rightarrow \infty} \widetilde{E}\left(\int_{0}^{T} \xi (t)\left(\left\|\widetilde{Z}\left(t \wedge \tau_{R}\right)\right\|^{2}-\|x\|_{H}^{2}\right) d t\right) .
\ee

Letting $R \rightarrow \infty$ in 
\eqref{ms2 p10b}-\eqref{ms2 p10c},
 by
  \eqref{tig2 2}-\eqref{tig2 3}
  and the Lebesgue domianted convergence thereom, 
  we obtain from \eqref{ms2 p13}
  that
$$
\widetilde{E}\left(2 \int_{0}^{T}\left(\xi(t) \sum_{j=1}^{J} \int_{0}^{t}\left(\left(\widetilde{A}_{j}(s)-A_{j}(s, \widetilde{X}(s)), \widetilde{Z}(s)-\widetilde{X}(s)\right)_{\left(V_{j}^{*}, V_{j}\right)}\right) d s\right) d t\right) 
$$
$$
+\widetilde{E}\left(\int_{0}^{T}\left(\xi(t) \int_{0}^{t}\|\widetilde{B}(s)-B(s, \widetilde{X}(s))\|_{\mathcal{L}_{2}(U, H)}^{2} d s\right) d t\right) \
$$
\be\label{ms2 p14}
\leq \lim _{R \rightarrow \infty} \liminf _{n \rightarrow \infty} \widetilde{E}\left(\int_{0}^{T} \xi(t)\left(\int_{0}^{t \wedge \tau_{R}}\left(g(s)+\varphi(\widetilde{X}(s))+\psi
(\widetilde{Z}_{n}(s)) \right)\left\|\widetilde{Z}_{n}(s)-\widetilde{X}(s)\right\|_{H}^{2} d s\right) d t\right) .
\ee
  
Next, we prove
\be\label{ms2 p15}
  \lim _{R \rightarrow \infty} \liminf _{n \rightarrow \infty} \widetilde{E}\left(\int_{0}^{T \wedge \tau_{R}}\left(|g(s)|+|\varphi(\widetilde{Z}(s))|+\left|\psi(\widetilde{Z}_{n}(s))\right|\right)\left\|\widetilde{Z}_{n}(s)-\widetilde{Z}(s)\right\|_{H}^{2} d s\right) =0,
\ee  
 for which we      first prove
\be\label{ms2 p16}
\lim _{n \rightarrow \infty} \widetilde{E}\left(\int_{0}^{T}|g(s)|\left\|\widetilde{Z}_{n}(s)-\widetilde{Z}(s)\right\|_{H}^{2} d s\right)=0.
\ee 
 By \eqref{tig2 2} and \eqref{ms1 p6}
  we find that for $\widetilde{\mathbb{P}}$-almost every $\omega$, there exists a number $c_{2}=c_{2}(\omega)>0$ such that for all $n \in \mathbb{N}$,
 \be\label{ms2 p17}
|g(s)|\left\|\widetilde{Z}_{n}(s)-\widetilde{Z}(s)\right\|_{H}^{2} \leq 2|g(s)|\left(\left\|\widetilde{Z}_{n}(s)\right\|_{H}^{2}+\|\widetilde{Z}(s)\|_{H}^{2}\right) \leq c_{2}|g(s)|, \quad   s \in[0, T] .
\ee
   Since $g \in L^{1}([0, T])$, by 
 \eqref{ms2 p7},
 \eqref{ms2 p17} and the Lebesgue dominated convergence theorem, we obtain, $\widetilde{\mathbb{P}}$-almost surely,
  \be\label{ms2 p18}
\lim _{n \rightarrow \infty} \int_{0}^{T}|g(s)|\left\|\widetilde{Z}_{n}(s)-\widetilde{Z}(s)\right\|_{H}^{2} d s=0 .
\ee

On the other hand, for every $p$ satisfying 
\eqref{main p}, by 
\eqref{tig2 1}-\eqref{tig2 2} we have
for all $n\in \N$,
$$
\widetilde{E}\left(\left(\int_{0}^{T}|g(s)|\left\|\widetilde{Z}_{n}(s)-\widetilde{Z}(s)\right\|_{H}^{2} d s\right)^{p}\right) 
$$
 \be\label{ms2 p19}
\leq 2^{2 p-1} \widetilde{E}\left(\sup _{t \in[0, T]}\left\|\widetilde{Z}_{n}\right\|_{H}^{2 p}+\sup _{t \in[0, T]}\|\widetilde{Z}\|_{H}^{2 p}\right)\left(\int_{0}^{T}|g(s)| d s\right)^{p} \le c_3, 
\ee
where $c_3= c_3 (T, p)>0$
is a constant.
  Then \eqref{ms2 p16}
   follows from  \eqref{ms2 p18}-\eqref{ms2 p19}
    and the Vitali theorem. 
    As a result of \eqref{ms2 p16},  we get
  \be\label{ms2 p20}
  \lim  _{n \rightarrow \infty} \widetilde{E}\left(\int_{0}^{T \wedge \tau_{R}}|g(s)|\left\|\widetilde{Z}_{n}(s)-\widetilde{Z}(s)\right\|_{H}^{2} d s\right)=0.
\ee

Similarly, 
 by
 \eqref{h2b},
  \eqref{tig2 2} and \eqref{ms1 p6}
  we find that for 
  every $R>0$ and $\widetilde{\mathbb{P}}$-almost every $\omega$, 
  there exists a number $c_{4}=c_{4}(R, \omega)>0$ such that for all $n \in \mathbb{N}$,
 \be\label{ms2 p20a}
 1_{\left(0, \tau_{R}\right]}(s)|\varphi(\widetilde{Z}(s))|\left\|\widetilde{Z}_{n}(s)-\widetilde{Z}(s)\right\|_{H}^{2}
   \le c_4 \sum_{j=1}^J
 (1+ \left\|\widetilde{Z}(s) \right \|_{V_j}^{q_j}),
 \quad s\in [0,T], 
\ee
  which along \eqref{tig2 2},
 \eqref{ms2 p7}
 and the Lebesgue dominated convergence theorem implies that, 
 $\widetilde{\mathbb{P}}$-almost surely,
  \be\label{ms2 p21a}
\lim _{n \rightarrow \infty} \int_{0}^{T}
1_{\left(0, \tau_{R}\right]}(s)|\varphi(\widetilde{Z}(s))|\left\|\widetilde{Z}_{n}(s)-\widetilde{Z}(s)\right\|_{H}^{2}
  d s=0 .
\ee
Furthermore,
 for every $p$ satisfying
\eqref{main p}, by \eqref{h2b}
and \eqref{tig2 1}-\eqref{tig2 2}
we have for all $n\in \N$,
 $$
  \widetilde{E}\left(\left(\int_{0}^{T} 1_{\left(0, \tau_{R}\right]}(s)|\varphi(\widetilde{Z}(s))|\left\|\widetilde{Z}_{n}(s)-\widetilde{Z}(s)\right\|_{H}^{2} d s\right)^{p}\right) 
  $$
    \be\label{ms2 p21}
  \leq c_{5} \widetilde{E}\left(\sup _{s \in[0, T]}\left\|\widetilde{Z}_{n}(s)\right\|_{H}^{2 p}+\sup _{s \in[0, T]}\left\|\widetilde{Z}(s)\right\|_{H}^{2 p}\right) \le  c_6 .
  \ee
 where $c_{6}=c_{6}(R, T, p)>0$.
  By \eqref{ms2 p21a}-\eqref{ms2 p21} and the Vitali theorem, we obtain
   \be\label{ms2 p22}
  \lim  _{n \rightarrow \infty} \widetilde{E}\left(\int_{0}^{T \wedge \tau_{R}}|\varphi(\widetilde{Z}(s))|\left\|\widetilde{Z}_{n}(s)-\widetilde{Z}(s)\right\|_{H}^{2} d s\right)=0 .
\ee
 By\eqref{h2c}  we have
$$
\widetilde{E}\left(\int_{0}^{T \wedge \tau_{R}}\left|\psi(\widetilde{Z}_{n}(s))\right|\left\|\widetilde{Z}_{n}(s)-\widetilde{Z}(s)\right\|_{H}^{2} d s\right) 
$$
$$
\leq \alpha_{1} \widetilde{E}\left(\int_{0}^{T \wedge \tau_{R}}\left(1+\left\|\widetilde{Z}_{n}(s)\right\|_{H}^{\alpha}\right)\left\|\widetilde{Z}_{n}(s)-\widetilde{Z}(s)\right\|_{H}^{2} d s\right) 
$$
$$
+\alpha_{1} \widetilde{E}\left(\sum_{j=1}^{J} \int_{0}^{T \wedge \tau_{R}}\left\|\widetilde{Z}_{n}(s)\right\|_{V_{j}}^{\theta_{j}}\left\|\widetilde{Z}_{n}(s)-\widetilde{Z}(s)\right\|_{H}^{2} d s\right)
$$
\be\label{ms2 p23}
+\alpha_{1} \widetilde{E}\left(\sum_{j=1}^{J} \int_{0}^{T \wedge \tau_{R}}\left\|\widetilde{Z}_{n}(s)\right\|_{V_{j}}^{\theta_{j}}\left\|\widetilde{Z}_{n}(s)\right\|_{H}^{\beta_{2, j}}\left\|\widetilde{Z}_{n}(s)-\widetilde{Z}(s)\right\|_{H}^{2} d s\right) . 
\ee

      Next, we prove every term on the right-hand side of (2.76) converges to zero as $n \rightarrow \infty$. 
    By   \eqref{h5b}  we have
   $$
\frac{1}{2}+\frac{1}{2} \max _{1 \leq j \leq J} \kappa_{j}<\frac{1}{2}+\min _{1 \leq j \leq J} \gamma_{1, j} \gamma_{2, j}^{-1},
$$
 which implies that there exists a number $p$ such that
\be\label{ms2 p24}
\frac{1}{2}+\frac{1}{2} \max _{1 \leq j \leq J} \kappa_{j}<p<\frac{1}{2}+\min _{1 \leq j \leq J} \gamma_{1, j} \gamma_{2, j}^{-1} .
\ee
By \eqref{h5c}  and \eqref{ms2 p24} we have
\be\label{ms2 p25}
p> 1+ {\frac 12} \alpha
\quad\text{and}\quad
p> 1+ {\frac 12} \beta_{2,j}
+ \theta_{j } q_j^{-1},
\quad \forall \
j=1,\cdots, J.
\ee
Let $q={\frac {2p}{2+\alpha}}$.
Then by \eqref{ms2 p25} we see 
$q>1$. By
\eqref{tig2 1}-\eqref{tig2 2} we get
for all $n\in \N$,
 $$ 
   \widetilde{E}\left(\int_{0}^{T  }
   \left (
   \left(1+\left\|\widetilde{Z}_{n}(s)\right\|_{H}^{\alpha}\right)\left\|\widetilde{Z}_{n}(s)-\widetilde{Z}(s)\right\|_{H}^{2}  
   \right )^q  ds \right) 
 $$
\be\label{ms2 p26}
\le c_7
\left (
1+ 
  \widetilde{E}\left( 
  \sup_{s\in [0,T]}
  \left\|\widetilde{Z}_{n}(s)
  \right\|_{H}^{2p}
  +
   \sup_{s\in [0,T]}
  \left\|\widetilde{Z}(s)
  \right\|_{H}^{2p}
  \right)
  \right )
  \le c_8,
 \ee
  where $c_8=c_8 (T,p)>0$ is a constant.
  
  It follows from \eqref{ms2 p7},
  \eqref{ms2 p26} and the
  Vitali theorem that
 \be\label{ms2 p27}   
 \lim _{n\to \infty} \widetilde{E}\left(\int_{0}^{T\wedge \tau_R   }
  \left(1+\left\|\widetilde{Z}_{n}(s)\right\|_{H}^{\alpha}\right)\left\|\widetilde{Z}_{n}(s)-\widetilde{Z}(s)\right\|_{H}^{2}  
   ds \right) 
=0.
\ee 
    For the second term
   on the right-hand side of
   \eqref{ms2 p23}, by \eqref{tig2 1}  we have
   for every $j=1,\cdots, J$,
    $$  \widetilde{E}\left(   \int_{0}^{T  }\left\|\widetilde{Z}_{n}(s)\right\|_{V_{j}}^{\theta_{j}}\left\|\widetilde{Z}_{n}(s)-\widetilde{Z}(s)\right\|_{H}^{2} d s\right)
   $$
   $$
   \le
    \widetilde{E}\left(
    \left (   \int_{0}^{T  }\left\|\widetilde{Z}_{n}(s)
    \right\|_{V_{j}}^{q_{j}}
    ds \right )^p
    \right )^{\frac {\theta_j }
    	{pq_j }}
    \left (  \widetilde{E}\left(
      \left (  \int_{0}^{T  }
    \left\|\widetilde{Z}_{n}(s)-\widetilde{Z}(s)\right\|_{H}^
    {\frac  {2q_j}{q_j- \theta_j }}
    	ds
     \right  )^{
    	\frac {p(q_j-\theta_j )}{
    		pq_j -\theta_j }}\right )
   \right )  ^{
     	\frac {pq_j-\theta_j }{
     			pq_j  }}
   $$
   \be\label{ms2 p28}
   \le c_9
    \left (  \widetilde{E}\left(
   \left (  \int_{0}^{T  }
   \left\|\widetilde{Z}_{n}(s)-\widetilde{Z}(s)\right\|_{H}^
   {\frac  {2q_j}{q_j- \theta_j }}
   ds
   \right  )^{
   	\frac {p(q_j-\theta_j )}{
   		pq_j -\theta_j }}\right )
   \right )  ^{
   	\frac {pq_j-\theta_j }{
   		pq_j  }},
   	\ee 
   	where $c_9=c_9 (T, p)>0$.
  Similar to
   \eqref{ms2 p18},
   by 
   \eqref{tig2 2},
   \eqref{ms1 p6} and \eqref{ms2 p7}
   and the Lebesgue dominated
   convergence theorem, we  get,
   $\widetilde{\p}$-almost surely,
    \be\label{ms2 p29}
    \lim_{n\to \infty}
      \int_{0}^{T  }
   \left\|\widetilde{Z}_{n}(s)-\widetilde{Z}(s)\right\|_{H}^
   {\frac  {2q_j}{q_j- \theta_j }}
   ds
    =0.
   \ee 
   Let $q=p-    \theta_j q_j^{-1} $.
   Then by
     \eqref{ms2 p25} we have
     $q>1$.
     Moreover,
     by \eqref{tig2 1}-\eqref{tig2 2}
     we get  for all $n\in \N$,
     $$
      \widetilde{E}\left(
   \left ( 
   \left ( 
    \int_{0}^{T  }
   \left\|\widetilde{Z}_{n}(s)-\widetilde{Z}(s)\right\|_{H}^
   {\frac  {2q_j}{q_j- \theta_j }}
   ds
   \right  )^{
   	\frac {p(q_j-\theta_j )}{
   		pq_j -\theta_j }}\right )^q
   \right ) 
     $$
     $$
     \le
     2^{\frac
     {2pqq_j}
     {pq_j-\theta_j}
     }
     T^{\frac
     {pq(q_j -\theta_j)}
     {pq_j -\theta_j}
     } \widetilde{E}
      \left (\sup_{t\in [0,T]}
      \left\|\widetilde{Z}_{n}(s)
      \right \|_H^{\frac
     {2pqq_j}
     {pq_j-\theta_j}
     }
      +\sup_{t\in [0,T]} \left \|   \widetilde{Z}(s)\right\|_{H}^{\frac
     {2pqq_j}
     {pq_j-\theta_j}
     }
     \right )
     $$
       \be\label{ms2 p30}
     =
     2^{2p}
   T^{p(1-\theta_j q_j^{-1})  }
      \widetilde{E}
      \left (\sup_{t\in [0,T]}
      \left\|\widetilde{Z}_{n}(s)
      \right \|_H^{2p
     }
      +\sup_{t\in [0,T]} \left \|   \widetilde{Z}(s)\right\|_{H}^{2p
     }
     \right )
     \le c_{10},
 \ee
     where $c_{10}=c_{10} (T, p)>0$
     is a constant.
     By \eqref{ms2 p29}-\eqref{ms2 p30}
     and the  Vitali theorem we obtain
     $$
      \lim_{n\to \infty}
      \widetilde{E}\left(
   \left (  
    \int_{0}^{T  }
   \left\|\widetilde{Z}_{n}(s)-\widetilde{Z}(s)\right\|_{H}^
   {\frac  {2q_j}{q_j- \theta_j }}
   ds
   \right  )^{
   	\frac {p(q_j-\theta_j )}{
   		pq_j -\theta_j }}\right ) 
   		=0,
     $$
     which along with
     \eqref{ms2 p28}
     implies that for every $j=1,\cdots, J$,
        \be\label{ms2 p31}  
        \lim_{n\to \infty}
         \widetilde{E}\left(  \int_{0}^{T \wedge \tau_{R}}\left\|\widetilde{Z}_{n}(s)\right\|_{V_{j}}^{\theta_{j}}\left\|\widetilde{Z}_{n}(s)-\widetilde{Z}(s)\right\|_{H}^{2} d s\right)
         =0.
         \ee

    For the last  term
   on the right-hand side of
   \eqref{ms2 p23}, 
  as for \eqref{ms2 p28}, by \eqref{tig2 1}  we have
   for every $j=1,\cdots, J$,
    $$  \widetilde{E}\left(   \int_{0}^{T  }\left\|\widetilde{Z}_{n}(s)\right\|_{V_{j}}^{\theta_{j}}
      \left\|\widetilde{Z}_{n}(s)\right\|_{H}^{\beta_{2, j}}
    \left\|\widetilde{Z}_{n}(s)-\widetilde{Z}(s)\right\|_{H}^{2} d s\right)
   $$
   \be\label{ms2 p32}
   \le c_{11}
    \left (  \widetilde{E}\left(
   \left (  \int_{0}^{T  }
   \left (
     \left\|\widetilde{Z}_{n}(s)\right\|_{H}^{\beta_{2, j}}
   \left\|\widetilde{Z}_{n}(s)-\widetilde{Z}(s)\right\|_{H}^2
   \right )^
   {\frac  {q_j}{q_j- \theta_j }}
   ds
   \right  )^{
   	\frac {p(q_j-\theta_j )}{
   		pq_j -\theta_j }}\right )
   \right )  ^{
   	\frac {pq_j-\theta_j }{
   		pq_j  }},
   	\ee 
   	where $c_{11}=c_{11} (T, p)>0$.
 As for 
   \eqref{ms2 p29}, by
   \eqref{tig2 2},
   \eqref{ms1 p6},
   \eqref{ms2 p7}
   and the Lebesgue dominated
   convergence theorem, we  get,
   $\widetilde{\p}$-almost surely,
    \be\label{ms2 p33}
    \lim_{n\to \infty}
      \int_{0}^{T  }
      \left ( \left\|\widetilde{Z}_{n}(s)
      \right \|_H^{\beta_{2,j}}
   \left\|\widetilde{Z}_{n}(s)-\widetilde{Z}(s)\right\|_{H}^2
   \right )^
   {\frac  {q_j}{q_j- \theta_j }}
   ds
    =0.
   \ee 
   Let $q={\frac {2(pq_j-\theta_j)}
   {q_j (2+\beta_{2,j})}
   } $.
   Then by
     \eqref{ms2 p25} we have
     $q>1$.
     Moreover,
     by \eqref{tig2 1}-\eqref{tig2 2}
     and Young's inequality, 
     we get  for all $n\in \N$,
     $$
      \widetilde{E}\left(
   \left ( 
   \left ( 
    \int_{0}^{T  }
    \left ( \left\|\widetilde{Z}_{n}(s)
      \right \|_H^{\beta_{2,j}}
   \left\|\widetilde{Z}_{n}(s)-\widetilde{Z}(s)\right\|_{H}^2
   \right )^
   {\frac  {q_j}{q_j- \theta_j }}
   ds
   \right  )^{
   	\frac {p(q_j-\theta_j )}{
   		pq_j -\theta_j }}\right )^q
   \right ) 
     $$
     $$
     \le
     2^{\frac
     {2pqq_j}
     {pq_j-\theta_j}
     }
     T^{\frac
     {pq(q_j -\theta_j)}
     {pq_j -\theta_j}
     } \widetilde{E}
        \left ( \sup_{t\in [0,T]}
      \left\|\widetilde{Z}_{n}(s)
      \right \|_H^{\frac
     {\beta_{2,j}pqq_j}
     {pq_j-\theta_j}}
      \left (
       \sup_{t\in [0,T]}
      \left\|\widetilde{Z}_{n}(s)
      \right \|_H^{\frac
     {2pqq_j}
     {pq_j-\theta_j}
     }
      +\sup_{t\in [0,T]} \left \|   \widetilde{Z}(s)\right\|_{H}^{\frac
     {2pqq_j}
     {pq_j-\theta_j}
     }  \right ) 
     \right )
     $$
       \be\label{ms2 p34}
    \le c_{12}
      \widetilde{E}
      \left (\sup_{t\in [0,T]}
      \left\|\widetilde{Z}_{n}(s)
      \right \|_H^{2p
     }
      +\sup_{t\in [0,T]} \left \|   \widetilde{Z}(s)\right\|_{H}^{2p
     }
     \right )
     \le c_{13},
 \ee
     where $c_{13}=c_{13} (T, p)>0$
     is a constant.
     By \eqref{ms2 p33}-\eqref{ms2 p34}
     and the  Vitali theorem we obtain
     $$
      \lim_{n\to \infty}
      \widetilde{E}\left(
   \left (  
    \int_{0}^{T  }
    \left (
    \left\|\widetilde{Z}_{n}(s)
      \right \|_H^{\beta_{2,j}}
   \left\|\widetilde{Z}_{n}(s)-\widetilde{Z}(s)\right\|_{H}^2
   \right )^
   {\frac  {q_j}{q_j- \theta_j }}
   ds
   \right  )^{
   	\frac {p(q_j-\theta_j )}{
   		pq_j -\theta_j }}\right ) 
   		=0,
     $$
     which along with
     \eqref{ms2 p32}
     implies that for every $j=1,\cdots, J$,
        \be\label{ms2 p35}  
        \lim_{n\to \infty}
         \widetilde{E}\left(  \int_{0}^
         {T \wedge \tau_{R}}
         \left\|\widetilde{Z}_{n}(s)\right\|_{V_{j}}^{\theta_{j}
         }
         \left\|\widetilde{Z}_{n}(s)
      \right \|_H^{\beta_{2,j}}
         \left\|\widetilde{Z}_{n}(s)-\widetilde{Z}(s)\right\|_{H}^{2} d s\right)
         =0.
         \ee
   
   By \eqref{ms2 p23},
   \eqref{ms2 p27},
      \eqref{ms2 p31} and
         \eqref{ms2 p35}, we obtain
               \be\label{ms2 p40}  
        \lim_{n\to \infty}
         \widetilde{E}\left(  \int_{0}^
         {T \wedge \tau_{R}}
         \left | \psi (\widetilde{Z}_{n}(s) )
         \right  |
         \left\|\widetilde{Z}_{n}(s)-\widetilde{Z}(s)\right\|_{H}^{2} d s\right)
         =0.
         \ee
 By \eqref{ms2 p20},
 \eqref{ms2 p22}
 and \eqref{ms2 p40},  we get
$$
  \lim _{n \rightarrow \infty} \widetilde{E}\left(\int_{0}^{T \wedge \tau_{R}}\left(|g(s)|+|\varphi(\widetilde{Z}(s))|+\left|\psi(\widetilde{Z}_{n}(s))\right|\right)\left\|\widetilde{Z}_{n}(s)-\widetilde{Z}(s)\right\|_{H}^{2} d s\right) =0,
$$
 which proves \eqref{ms2 p15}.

   Given $\delta\in (0,1)$
   and $v\in V$, let
   $\widetilde{X}
   =\widetilde{Z} -\delta v$. Then by
   the argument of  \eqref{ms2 p14}
   we get,
    for any nonnegative 
 $ \xi \in {L^{\infty}([0, T] \times 
 \widetilde{\Omega})}$,
$$ 2
\widetilde{E}\left(  \int_{0}^{T}\left(\xi(t) \sum_{j=1}^{J} \int_{0}^{t} \left(\widetilde{A}_{j}(s)-A_{j}(s, \widetilde{Z}(s) -\delta v ),
 \delta v  \right)_{\left(V_{j}^{*}, V_{j}\right)}  d s\right) d t\right) 
$$
 \be\label{ms2 p42}
\leq \lim _{R \rightarrow \infty} \liminf _{n \rightarrow \infty} \widetilde{E}\left(\int_{0}^{T} \xi(t)\left(\int_{0}^{t \wedge \tau_{R}}
(g(s)+\varphi(\widetilde{Z}(s)
-\delta v )+\psi
(\widetilde{Z}_{n}(s)) 
)
 \|\widetilde{Z}_{n}(s)-\widetilde{Z}(s) +\delta v  \|_{H}^{2} d s\right) d t\right) ,
\ee
   where
   $\tau_R =\tau_R(\widetilde{Z})$.
   For the  right-hand side of \eqref{ms2 p42}
   we have
  $$
   \widetilde{E}\left(\int_{0}^{T} \xi(t)\left(\int_{0}^{t \wedge \tau_{R}}
(g(s)+\varphi(\widetilde{Z}(s)
-\delta v )+\psi
(\widetilde{Z}_{n}(s)) 
)
 \|\widetilde{Z}_{n}(s)-\widetilde{Z}(s) +\delta v  \|_{H}^{2} d s\right) d t\right)
   $$
   $$
   \le
  c_{14}
   \widetilde{E} 
   \left(\int_{0}^{T \wedge \tau_{R}}
 \left | g(s)+\varphi(\widetilde{Z}(s)
-\delta v )+\psi
(\widetilde{Z}_{n}(s)) 
\right | 
 \|\widetilde{Z}_{n}(s)-\widetilde{Z}(s) 
  \|_{H}^{2} d s \right)
      $$
     \be\label{ms2 p43}
      + 
  c_{14}\delta^2 \|v \|_H^2
   \widetilde{E} 
   \left(\int_{0}^{T }
\left | g(s)+\varphi(\widetilde{Z}(s)
-\delta v )+\psi
(\widetilde{Z}_{n}(s)) 
 \right |
   d s \right),
   \ee
 where    $ c_{14}=
  2T\|\xi\|_{L^{\infty}([0, T] \times 
 \widetilde{\Omega})}$.
Similar to \eqref{ms2 p15},
the first term on the right-hand side of
\eqref{ms2 p43}
converges to zero as $n\to \infty$.
On the other hand, by \eqref{h2b}-\eqref{h2c}
and \eqref{tig2 1}-\eqref{tig2 2} we find that
there exists a constant
$c_{15}=
c_{15} (T, p, v )>0$ such that
for all $\delta\in (0,1)$  and $n\in \N$,
      \be\label{ms2 p44}
  \widetilde{E} 
   \left(\int_{0}^{T }
| g(s)+\varphi(\widetilde{Z}(s)
-\delta v )+\psi
(\widetilde{Z}_{n}(s)) 
| 
   d s \right)
   \le c_{15}.
\ee
By \eqref{ms2 p42}-\eqref{ms2 p44}
   we get,
    for any nonnegative 
 $ \xi \in {L^{\infty}([0, T] \times 
 \widetilde{\Omega})}$,
 $\delta\in (0,1)$ and $v\in V$,
$$ 2
\widetilde{E}\left(  \int_{0}^{T}\left(\xi(t) \sum_{j=1}^{J} \int_{0}^{t} \left(\widetilde{A}_{j}(s)-A_{j}(s, \widetilde{Z}(s) -\delta v ),
 \delta v  \right)_{\left(V_{j}^{*}, V_{j}\right)}  d s\right) d t\right) 
 \le c_{14} c_{15} \delta^2 \| v\|_H^2.
$$
   First diving both sides of
   the above inequality by $\delta$,
    and then   
      taking the limit as $\delta \to  0$,
      by ({\bf H1})
      we obtain,
        for any nonnegative 
 $ \xi \in {L^{\infty}([0, T] \times 
 \widetilde{\Omega})}$ and $v\in V$, 
\be\label{ms2 p50}
\widetilde{E}\left(  \int_{0}^{T}\left(\xi(t) \sum_{j=1}^{J} \int_{0}^{t} \left(\widetilde{A}_{j}(s)-A_{j}(s, \widetilde{Z}(s)   ),
  v  \right)_{\left(V_{j}^{*}, V_{j}\right)}  d s\right) d t\right) 
 \le 0  .
\ee
Replacing $v$ by $-v$,  we get from
\eqref{ms2 p50} that  for any nonnegative 
      $ \xi \in {L^{\infty}([0, T] \times 
 \widetilde{\Omega})}$ and $v\in V$, 
$$
\widetilde{E}\left(  \int_{0}^{T}\left(\xi(t) \sum_{j=1}^{J} \int_{0}^{t} \left(\widetilde{A}_{j}(s)-A_{j}(s, \widetilde{Z}(s)   ),
  v  \right)_{\left(V_{j}^{*}, V_{j}\right)}  d s\right) d t\right) 
= 0  ,
$$ 
which implies that
$
 \sum_{j=1}^J\widetilde{A}_{j}
 =\sum_{j=1}^J
  A_{j}(\cdot, \widetilde{Z}   )$
  a.e. on $[0,T]\times \widetilde{\Omega}$.
  
  By Steps (ii) and (iii), we see that
   $\left(\widetilde{\Omega}, \widetilde{\mathcal{F}},(\widetilde{\mathcal{F}}_{t})_{t \in[0, T]}, \widetilde{\mathbb{P}}, \widetilde{Z}, \widetilde{W}\right)$ is a martingale
   solution of \eqref{sde1}-\eqref{sde2}.
   Moreover,
   by the standard argument
   (see,  e.g., \cite{roc1}), one can verify 
   that the pathwise uniqueness   holds
   for the solutions of 
      \eqref{sde1}-\eqref{sde2}
      under ({\bf H1})-({\bf H5}),
      which along with the 
      Yamada-Watanabe
 theorem implies the existence and
 uniqueness of  strong probabilistic solutions
 as stated in Theorem \ref{main}.

      \begin{rem}\label{2.8a}
      The  pathwise uniform estimate
      \eqref{ms1 p6} is crucial
      for proving the convergence of 
      approximate solutions,
       and frequently used 
       for verifying the uniform
       integrability of
       a sequence of functions
       involving $\{ \widetilde{Z}_n\}_{n=1}^\infty$,
       see,  e.g.,
       the proof of
       \eqref{ms2 p17},
       \eqref{ms2 p20a},
       \eqref{ms2 p29}
       and \eqref{ms2 p33}.
        Without \eqref{ms1 p6},   it is
      difficult to  establish
      the 
      uniform  integrability of those sequences.
  To  obtain the  pathwise   uniform estimate
      \eqref{ms1 p6}, we  have to   prove the
      tightness of the  sequence
      $\{{Z}_n\}_{n=1}^\infty$
      in $L^\infty_{w^*} (0,T; H)$ 
       and then employ the 
       Skorokhod-Jakubowski
       representation theorem 
       given by  Proposition \ref{prop_sj}
       in a topological space  
      rather   a metric space. 
       The classical Skorokhod
       representation theorem
       in a metric space  does not apply
       to   $L^\infty_{w^*} (0,T;  H)$.
       \end{rem}

\section{Well-posedness of   fractional 
$p$-Laplace   equations}
\setcounter{equation}{0}

In this section, we apply Theorem \ref{main}
to investigate the existence 
and uniqueness of  
solutions of the fractional
 stochastic $p$-Laplace
  equation
\eqref{lap1}-\eqref{lap3}
defined in a bounded domain
$\o$ in $\R^n$
driven by superlinear transport  noise.

We start with the definition
of the fractional $p$-Laplace operator.
Given $s\in (0,1)$  and $2\leq p<\infty$,
the operator
$ (-\Delta )^s_p$ is defined by,
  for $x\in\mathbb{R}^n$,
$$  (-\Delta)_{p}^su(x)
= 
C(n,p,s)\ \mbox{P.V.}
\int_{\mathbb{R}^n}\frac{|u(x)-u(y)|^{p-2}\big(u(x)-u(y)\big)}{|x-y|^{n+ps}}dy,
$$ 
 provided the limit exists, where P.V. indicates the principal value of the integral, and
   $C(n,p,s)$ is a positive number
   depending on $n$, $p$ and $s$
     given by
\be\label{constpl}
C(n,p,s)=\frac{s4^s\Gamma\big(\frac{ps+p+n-2}{2}\big)}{\pi^\frac{n}{2}\Gamma(1-s)},
\ee
with 
   $\Gamma$
   being the   Gamma function.

  The fractional Sobolev space $W^{s,p}(\mathbb{R}^n)$  is given by 
$$
W^{s,p}(\mathbb{R}^n)= 
\bigg\{u\in L^p(\mathbb{R}^n):\int_{\R^n}
\int_{\R^n}\frac{|u(x)-u(y)|^p}{|x-y|^{n+ps}}dxdy<\infty  \bigg\},
$$
with  norm
$$
\|u\|_{W^{s,p}} =\bigg(\int_{\R^n}|u(x)|^pdx+\int_{\R^n}\int_{\R^n}\frac{|u(x)-u(y)|^p}{|x-y|^{n+ps}}dxdy \bigg)^\frac{1}{p},\ \ u\in W^{s,p}(\mathbb{R}^n).
$$
The  Gagliardo semi-norm of $W^{s,p}(\mathbb{R}^n)$  is given by
 $$
 [u]_{{W}^{s,p}}
 =\bigg(\int_{\R^n}\int_{\R^n}\frac{|u(x)-u(y)|^p}{|x-y|^{n+ps}}dxdy \bigg)^\frac{1}{p},\ \ u\in W^{s,p}(\R^n).
 $$

From now on,   we  denote by
$H= \{ u\in 
  L^2(\R^n):  u=0 \ \text{a.e.  on }  \R^n \setminus \o
  \}$ with norm
  $\| \cdot \|_H$ and inner product
  $(\cdot, \cdot )_H$,
 $V_1=   \{ u\in 
  W^{s,p} (\R^n):  u=0 \ \text{a.e.  on }  \R^n \setminus \o
  \}$ with $p\ge 2$, 
  $V_2 =  \{ u\in 
  L^q(\R^n):  u=0 \ \text{a.e.  on }  \R^n \setminus \o
  \}$
  with $q\ge 2$ and   $V_3=H$.
  Let
  $V= V_1 \bigcap V_2 \bigcap V_3=V_1 \bigcap V_2 $.
Then we  have the embeddings
$V \subseteq  H  \equiv  H^*  \subseteq  V^*$.
Note that the embedding 
   $V \subseteq  H$   is compact
   since $\o$ is bounded in $\R^n$.
   
   Recall Poincare's inequality on $V_1$:
   there exists a constant $\lambda 
   =\lambda(s, p, n,
   \o)>0$ such that
\be\label{pi}
   [u]^p_{W^{s,p}
   (\R^n)} \ge \lambda   \| u \|^p_{L^p(\R^n)},
   \quad \forall \  u \in V_1,
\ee
 which shows that
  $ [\cdot]_{W^{s,p}
   (\R^n)}$ is also a norm on $V_1$
   and equivalent to
  $ \| \cdot \|_{W^{s,p}
   (\R^n)}$.
   In the sequel, we consider 
   $ [\cdot]_{W^{s,p}
   (\R^n)}$ as the norm of $V_1$ and
   write
   $\| u \|_{V_1}
   =   [u]_{W^{s,p}
   (\R^n)}$ for  $u\in V_1$.

   We first consider
   the existence and uniqueness of
   solutions of 
   \eqref{lap1}-\eqref{lap3}
   when $f$ is a general monotone function,
    and  then  improve the result  when
   $f$ satisfies a stronger monotonicity condition.

\subsection{Fractional 
$p$-Laplace   equations with a general
monotonicity  drift}

In this subsection, we prove the existence
and uniqueness of solutions of \eqref{lap1}-\eqref{lap3} when the nonlinearity  $f$ 
is monotone and dissipative.
More precisely, we now   assume
 $f: \R \times \R^n  \times \R
 \to \R$ is  continuous   such that
 for all $t, u, u_1, u_2 \in \R$ and $x\in \R^n$,
 $f(t,x, 0) =0$ and
  \be\label{f1}
 (f(t,x, u_1) - f(t,x, u_2))(u_1-u_2)
 \le 0,
\ee
 \be\label{f2}
 f(t,x, u) u \ \le -  \delta_1 |u|^q +  \varphi_1 (t,x),
 \ee
  \be\label{f3}
 | f(t,x, u) |  \ \le \delta_2  |u|^{q-1} +  \varphi_2 (t,x),
 \ee
 where $\delta_1>0$,
 $\delta_2>0$ and $q \ge 2$ are constants,
   $\varphi_1\in L^1([0,T] \times \o) $, and 
  $\varphi_2 \in L^{\frac q{q-1}}
 ([0,T] \times \o)$.

 For the nonlinear term 
 $h$, we assume   
 $h: \R \times \R^n  \times \R
 \to \R$ is   continuous   such that
 for all $t, u, u_1, u_2 \in \R$ and $x\in \R^n$,
 $h(t,x, 0) =0$ and 
 \be\label{h1}
 |h(t,x, u_1) -  h(t,x, u_2)|
 \le \varphi_3 (t,x)
 |u_1-u_2|  ,
 \ee
where $\varphi_3\in L^1(0,T; L^\infty(\o))$
 for every $T>0$.

For every $i\in \N$, 
let
        $\sigma_i  :
      \R \times \R^n \times \R
      \to \R$  be a mapping given  by   
  \be\label{sig1}
     \sigma_i  (t,x, u)
     = \sigma_{1,i}   (t,x)
     +   
         \sigma_{2, i } ( u)  
        , 
       \quad \forall \ t\in \R, \ x\in \R^n, \ u\in \R,
\ee
     where      
   $\sigma_{1,i}: [0,T]\to H$ for every $T>0$
    such that
   \be\label{sig2}
   \sum_{i=1}^\infty  \| \sigma_{1,i}\|^2
   _ {L^2(0,T; H )}
   <\infty.
   \ee 
   For every $i\in \N$,
    assume   
    $\sigma_{2,i}: \R
    \to \R$ is continuous such that
    $\sigma_{2,i} (0) =0$ and 
  there exist  positive numbers $\beta_i$ and $\gamma_i$
   such that   for all 
    $u, u_1, u_2   \in \R$,
    \be\label{sig3}
   |\sigma_{2,i}  (u)  |^2
   \le \gamma_i    +  \beta_i  |u|^{p_1}  ,
   \ee
    \be\label{sig4}
| \sigma_{2,i}  (u_1) 
-
\sigma_{2,i}  (u_2) |^2
\le \gamma_i
(1+ |u_1|^{p_1-2}
+|u_2|^{p_1-2})
|u_1-u_2|^2 , 
 \ee
    where $p_1 \in [2, p]$ and
  \be\label{sig5}
  \sum_{i=1}^\infty  ( 
    \beta_i  + \gamma_i  ) <\infty .
     \ee

Note that  
     $   \sigma_{2,i}$
   has     superlinear growth
   when $p_1>2$. 
  Given $v\in L^{p_1}
  (\R^n)
  \bigcap H$  and $t\in \R$, define an operator
  $B (t, v): l^2\to H$ by
\be\label{sig6}
  B (t,v) (u) (x)
  =\sum_{i=1}^\infty
  \sigma_i (t,x, v(x)) u_i
   \quad \forall \ 
   u=\{u_i\}_{i=1}^\infty \in l^2,
   \ \ x\in \R^n.
\ee
It follows from 
  \eqref{sig1}-\eqref{sig3}
  and \eqref{sig5}-\eqref{sig6}    that 
  $B (t, v): l^2\to H$ is a   Hilbert-Schmidt 
  operator 
   with norm
 $$
 \|  B (t,v)   \|^2_{\call_2 (l^2, H)}
 =
 \sum_{i=1}^\infty
 \|  \sigma_i (t,\cdot, v ) \|^2_H
  $$ 
   $$
   \le 2 
 \sum_{i=1}^\infty
  \int_{\o} \left (
   |  \sigma_{1,i}  (t,x)|^2 
     +      | \sigma_{2,i} ( v(x) ) 
       |^2 \right ) dx
  $$
 $$
  \le
  2 \sum_{i=1}^\infty 
  \|\sigma_{1, i} (t) \|^2_ {H}
  + 2 |\o|
  \sum_{i=1}^\infty \gamma_i 
  +2\sum_{i=1}^\infty
  \beta_i\int_{\o}  |v(x)|^{p_1} dx
 $$
  \be\label{sig7}
  \le
  2 \sum_{i=1}^\infty 
  \|\sigma_{1, i} (t) \|^2_ {H}
  + 2 |\o|
  \sum_{i=1}^\infty \gamma_i 
  +2|\o|^{\frac {p-p_1}p} \sum_{i=1}^\infty
  \beta_i 
  \left (\int_{\o}  |v(x)|^{p} dx
  \right )^{\frac {p_1}p} 
\ee
 where $|\o|$ is the volume of $\o$.

Similarly, by   \eqref{sig1} and   
   \eqref{sig4}-\eqref{sig6}    we have, 
     for all 
 $u_1,u_2\in
L^{p_1}(\mathbb{R}^n)\bigcap H$,
$$
 \|B (t,  u_1)-
B (t,  u_2)\|^2_{\call_2(l^2,  H)}
 = \sum_{i=1}^\infty 
 \int_{\o} |\sigma_{2,i} (
 u_1(x))-\sigma_{2,i}(u_2(x))|^2dx
$$
  \be\label{sig8}
\leq   \sum_{i=1}^\infty \gamma_i 
  \int_{\o} 
   \big( 1+ |u_1(x)|^{p_1-2}  
+|u_2(x)|^{p_1-2}
\big)|u_1(x)-u_2(x)|^2dx
\ee

Let    $A_1: V_1 \to V_1^*$ be
the operator given by:
for every    $v, u\in V_1$, 
\be\label{cA1a}
 (A_1v, \  u)_{(V_1^*, V_1)}
 =
  -{\frac 12}  
C(n,p,s)\  
\int_{\R^n }
\int_{\R^n}
\frac{|v(x)-v(y)|^{p-2}\big(v(x)-v(y)\big)
\big(u(x)-u(y)\big)
}
{|x-y|^{n+ps}}dx dy.
\ee
Note that
\be\label{cA1}
A_1:  V_1 \to V_1^*
\ \text{  is hemicontinuous}.
\ee
 Let  $A_2(t) $ be 
 the Nemytskii
 operator associated with $f(t, \cdot, \cdot)$;
 that is, 
 $A_2(t, v) (x)  = f(t, x, v(x))$ for all
 $v\in V_2$ and $x\in \R^n$.
Then by   \eqref{f3}  we see  that for $t\in [0,T]$,
\be\label{cA2}
A_2(t, \cdot) :
V_2 \to V_2^* \ \text{ is continuous}.
\ee
Let  $A_3(t) $ be the Nemytskii
 operator associated with $h(t, \cdot, \cdot)$
 on $V_3$. 
Then by   \eqref{h1}  we see  that for $t\in [0,T]$,
\be\label{cA3}
A_3(t, \cdot) :
V_3 \to V_3^* \ \text{ is continuous}.
\ee
Denote by 
 $A(t)=A_1 +A_2(t) +A_3(t)$
 for  all $t\in [0,T]$.
 Then $A(t)$ is an operator
 from  $ V$ to $ V^*$.

 In terms of the above notation,
 the stochastic $p$-Laplace equation
 \eqref{lap1} can be rewritten in the
 abstract form: 
$$
  du (t)
   = A(t, u(t)) dt  
   +  B  (t, u(t))    {dW},
$$
where $W$ is a cylindrical 
Wiener process
in $l^2$.

%
%
%

The main result of this subsection is given below.

\begin{thm}\label{ma1pla}
	If  \eqref{f1}-\eqref{sig5} hold,
	$sp>n$ and
	\be\label{ma1pla 1}
	 \sum_{i=1}^\infty
	\beta_i <{\frac 16}  \lambda C(n,p, s),
	\ee
	where $\lambda$ is the constant
	in \eqref{pi},
	$C(n,p,s )$ is 
	the number in \eqref{constpl},
	and $\{\beta_i\}_{i=1}^\infty$
	is the sequence in \eqref{sig3}.
Then for every $u_0 \in H$,
	system \eqref{lap1}-\eqref{lap3}
	has a unique solution
	in the sense of Definition \ref{dsol}. 
	Moreover, the uniform estimates
	given by \eqref{main 1} are valid.
\end{thm}

\begin{proof}
Under conditions  \eqref{f1}-\eqref{sig5}, we
  will 
  verify all assumptions 
	of Theorem \ref{main} are fulfilled.
	First, by  \eqref{cA1}-\eqref{cA3}, we find
	that $A: V\to V^*$ is 
   hemicontinuous; that is,
	$A$ satisfies 
 {\bf (H1)}.

 Next, we verify  {\bf (H2)}.
 Recall the inequality: for all $s_1, s_2 \in \R$
 and $p\ge 2$,
 \be \label{ma1pla p1}
 (|s_1|^{p-2} s_1-
 |s_2|^{p-2} s_2) (s_1- s_2)
 \ge 2^{1-p} |s_1 -s_2 |^p.
 \ee
 By  \eqref{ma1pla p1} we have, 
 for all $u, v\in V_1$,
  \be \label{ma1pla p2}
	2( A_1(u) -A_1 (v),
	u-v)_{(V_1^*, V_1)}
	\le - 2^{1-p} C(n,p , s)
	\| u -v \|_{V_1}^p.
	\ee
	By 
	  \eqref{f1}, \eqref{h1} and
	  \eqref{ma1pla p2} 
	   we get,
	for $ t\in [0,T]$ and $u, v\in V$,
	   \be \label{ma1pla p3}
	2( A(t,u) -A(t,v),
	u-v)_{(V^*, V)}
	\le  - 2^{1-p} C(n,p , s)
	\| u -v \|_{V_1}^p +
2	\| \varphi_3(t)\|_{L^\infty(\o)}
	\| u-v \|_H^2.
\ee
Since $p>{\frac ns}$, by the embedding theorem,
we find that there exists $c_1= c_1 (n, p,  s,
\o)>0$ such that 
 	   \be \label{ma1pla p4}
 \|u \|_{L^\infty (\o)}
 \le c_1 \| u \|_{V_1},
 \quad \forall \ u \in V_1.
\ee
	By \eqref{sig8} and \eqref{ma1pla p4} we get,
	for $t\in [0,T]$ and
  $u, v\in V$,
	   \be \label{ma1pla p5}
\|  B(t,u) -B(t,v) \|^2_{\call_2(l^2, H)}
\le
\sum_{i=1}^\infty \gamma_i
\left (
1+ \| u\|_{L^\infty(\o)}^{p_1-2}
+ \|  v\|_{L^\infty(\o)}^{p_1-2}
\right ) \| u-v \|^2_H 
\ee
	Since $p_1\in [2, p]$, by
	\eqref{ma1pla p4}-\eqref{ma1pla p5}
	we get, 
	for $t\in [0,T]$ and
  $u, v\in V$,
	   \be \label{ma1pla p6}
\|  B(t,u) -B(t,v) \|^2_{\call_2(l^2, H)}
\le
c_2
\left (
1+ \| u\|_{V_1}^{p-2}
+ \|  v\|_{V_1}^{p-2}
\right ) \| u-v \|^2_H ,
\ee
	for some
	$c_2=c_2 (n,p, s, \o)>0$.
	By    
	  \eqref{ma1pla p3} 
	  and  \eqref{ma1pla p6} 
	   we have,
	for $ t\in [0,T]$ and $u, v\in V$,
	$$
	2( A(t,u) -A(t,v),
	u-v)_{(V^*, V)}
	+\|  B(t,u) -B(t,v) \|^2_{\call_2(l^2, H)}
	$$
	   \be \label{ma1pla p7}
	\le   \left (   2
	\| \varphi_3(t)\|_{L^\infty(\o)}
  +c_2 
+c_2  \| u\|_{V_1}^{p-2}
+c_2 \|  v\|_{V_1}^{p-2}
\right ) \| u-v \|^2_H,
\ee
	 where $\varphi_3\in L^1(0,T;
	L^\infty(\o))$.
	By \eqref{ma1pla p7},  after
	simple calculations, we infer that
	 {\bf (H2)}  is satisfied
	 with the following parameters:
	 $g=2\| \varphi_3(\cdot)\|_{   
	L^\infty(\o) }$ and 
	 	   \be \label{ma1pla p7a}
	 q_1 =p, \ q_2 =q,\
	 q_3= 2,\
	 \alpha =0,\
	 \theta_1= p-2,\
	 \theta_2=\theta_3
	 =\beta_{1, j}=\beta_{2,j} =0,
	 \quad  \ j=1,2,3.
\ee

	  We now  verify  {\bf (H3)}.
 By  \eqref{f2},  \eqref{h1}
	 and 
	 \eqref{cA1a}, 
	 we have 
  for $t\in [0,T]$  and $v\in V$,
 	\be\label{ma1pla p8}
	 ( A(t,v), v)_{(V^*,V)} 
	\le
	-{\frac 12} C(n, p, s)
	 \| v\|^p _{V_1}
	-\delta_1 \| v\|^q_{V_2}
	+  \|\varphi_1 (t)\|_{L^1(\o)}
	+ 
	\|\varphi_3 (t)\|_{L^\infty(\o)}\|v\|^2_H,
	 \ee
	where $\varphi_1\in L^1(0,T; L^1(\o))$
	and
	$\varphi_3 \in L^1(0,T; L^\infty(\o))$.
	By \eqref{ma1pla p8} we find that
	{\bf (H3)} is satisfied with the parameters:
 	   \be \label{ma1pla p8a}
	\gamma_{1,1} = {\frac 12} C(n, p, s), \
	\gamma_{1,2}= \delta_1,
	\ \gamma_{1,3} = 1.
\ee
	
	On the other hand,
		{\bf (H4)} follows straightforward
		from \eqref{f3}, \eqref{h1} and
		\eqref{cA1a} with the parameters
		given by \eqref{ma1pla p7a}.
		It remains to verify  	{\bf (H5)}.
		
		By \eqref{pi}, \eqref{sig7}
		and Young's inequality,
		we find that there exists
		a constant $c_3=c_3
		(n, p,  s,   \o)>0$
		 such that  
		for $t\in [0,T]$
		and $v\in V$,
		$$
		\| B(t, v)\|^2_{\call_2(l^2, H)}
		\le  2 \sum_{i=1}^\infty 
		 \|\sigma_{1,i}  (t) \|^2_H
		 + c_3
		 + 2\sum_{i=1}^\infty \beta_i
		 \int_\o |v(x)|^p dx
		 $$
		 $$
		\le  2 \sum_{i=1}^\infty 
		 \|\sigma_{1,i}  (t) \|^2_H
		 + c_3
		 + 2\lambda^{-1} \sum_{i=1}^\infty \beta_i
	\| v \|^p_{V_1},
		 $$
	which along with 	\eqref{sig2}
	shows that  	{\bf (H5)} is satisfied with
	the parameters:
	   \be \label{ma1pla p9}
	\gamma_{2,1} =  2\lambda^{-1} \sum_{i=1}^\infty \beta_i, \
	\gamma_{2,2}=0,
	\ \gamma_{2,3} = 0.
\ee

By  \eqref{ma1pla p7a},
	\eqref{ma1pla p8a} and \eqref{ma1pla p9}
we find that, 
in the present case, the constants in \eqref{h5c}
are given by
$$
\kappa_1 = 3- 4p^{-1}, \
\kappa_2=\kappa_3 =1,
$$
and  the condition \eqref{h5b}
reduces to:
 	   \be \label{ma1pla p10}
{\frac 1p}
> {\frac 34}
-{\frac {\lambda C(n,p,s)}{8\sum_{i=1}^\infty
\beta_i}}.
\ee
By \eqref{ma1pla 1} we have
 ${\frac 34}
-{\frac {\lambda C(n,p,s)}{8\sum_{i=1}^\infty
\beta_i}}<0$, and hence
\eqref{ma1pla p10} is fulfilled.
By \eqref{ma1pla p6} we find that
$B$ is continuous from $V$ to $\call_2(U,H)$ and
  satisfies \eqref{h5d}, and hence satisfies \eqref{h5c}.
	This completes the proof. 
 \end{proof}

Notice that Theorem \ref{ma1pla}
is valid only under the condition
$p>{\frac ns}$ which excludes the
standard fractional Laplace operator
for $p=2$ when $n>1$.
In the next subsection, we will remove this
restriction when $f$ satisfies a strong
monotonicity condition.

\subsection{Fractional 
$p$-Laplace   equations with a 
strong 
monotonicity  drift}

In this subsection, we 
improve Theorem \ref{ma1pla} by removing
the assumption
$p>{\frac ns}$. To that end, we will require
the nonlinearity $f$ satisfy  a stronger
monotone condition than
\eqref{f1}. More precisely, 
throughout this subsection,  
 we further assume    that 
$f$ satisfies:
for all $t, u_1, u_2\in \R$ and $x\in \R^n$,
 \be\label{ff1}
\left (
f(t,x, u_1) -f(t,x, u_2)\right )
(u_1-u_2)
 \le -\delta_3 \left ( |u_1|^{q-2}
+ |u_2|^{q-2} \right ) (u_1 -u_2)^2,
 \ee
where $\delta_3>0$ and $q\ge 2$  are  constants.
Note that   if $f(t,x, u) =- |u|^{q-2} u$, then
\eqref{ff1} is fulfilled.

In this subsection, we
   assume the exponent
$p_1$ in \eqref{sig3}-\eqref{sig4}
belongs to $[2, q)$ instead of $[2, p]$.
In this case, by the argument of \eqref{sig7}
we find that
for every $v\in L^{p_1}
  (\R^n)
  \bigcap H$  and $t\in \R$,   
   \be\label{sig77a}
 \|  B (t,v)   \|^2_{\call_2 (l^2, H)}
   \le
  2 \sum_{i=1}^\infty 
  \|\sigma_{1, i} (t) \|^2_ {H}
  + 2 |\o|
  \sum_{i=1}^\infty \gamma_i 
  +2|\o|^{\frac {q-p_1}q} \sum_{i=1}^\infty
  \beta_i 
  \left (\int_{\o}  |v(x)|^{q} dx
  \right )^{\frac {p_1}q} .
\ee

The main result of this section is given below.

\begin{thm}\label{ma2pla}
	If  \eqref{f2}-\eqref{sig5}
	and \eqref{ff1} hold 
  and
	\be\label{ma2pla 1} 
	2\le p_1<  q, \quad 
 \sum_{i=1}^\infty
\beta_i  < \delta_1, \quad
   \sum_{i=1}^\infty
	\gamma_i  \le   2 \delta_3,
	\ee
	where  $p_1$, $ \beta_i ,$
	and   $ \gamma_i  $
	are  the numbers  in 
	\eqref{sig3}-\eqref{sig4}.
Then for every $u_0 \in H$,
	system \eqref{lap1}-\eqref{lap3}
	has a unique solution
	in the sense of Definition \ref{dsol}. 
	Moreover, the uniform estimates
	given by \eqref{main 1} are valid.
\end{thm}

\begin{proof}
In the present case,  we 
  only need to verify {\bf (H2)}
and {\bf (H5)} of  Theorem \ref{main}
since {\bf (H1)}, {\bf (H3)} and {\bf (H4)}
are the same as   in Theorem \ref{ma1pla}.

  Since $p_1\in [2,q]$,
by  \eqref{sig8} and Young's inequality  we have,
     for  $t\in [0,T]$ and 
 $u_1,u_2\in
L^{p_1}(\mathbb{R}^n)\bigcap H$,
$$
 \|B (t,  u_1)-
B (t,  u_2)\|^2_{\call_2(l^2 H)}
  \leq   \sum_{i=1}^\infty \gamma_i 
  \int_{\o} 
   \big( 1+ |u_1(x)|^{p_1-2}  
+|u_2(x)|^{p_1-2}
\big)|u_1(x)-u_2(x)|^2dx
$$
  \be \label{ma2pla p1a}
\le
c_1\|u_1-u_2\|^2_H  + 
   \sum_{i=1}^\infty \gamma_i 
  \int_{\o} 
   \big(   |u_1(x)|^{q-2}  
+|u_2(x)|^{q-2}
\big)|u_1(x)-u_2(x)|^2dx,
\ee
for some constant $c_1=c_1( p_1, q, \o)>0$.
On the other hand,  by  
   \eqref{h1},
	  \eqref{ma1pla p2}  and \eqref{ff1}
	   we get,
	for $ t\in [0,T]$ and $u, v\in V$,
	$$
	2( A(t,u) -A(t,v),
	u-v)_{(V^*, V)}
	$$
	$$
  \le   
	-2^{1-p} C(n,p,s) \| u-v\|^p_{V_1}
	 +
	2\| \varphi_3(t)\|_{L^\infty(\o)}
	\| u-v \|_H^2
	$$
		   \be \label{ma2pla p3}
	 -2\delta_3  \int_{\o} 
   \big(   |u_1(x)|^{q-2}  
+|u_2(x)|^{q-2}
\big)|u_1(x)-u_2(x)|^2dx .
\ee
It follows from \eqref{ma2pla 1}-\eqref{ma2pla p3}
that
   for $ t\in [0,T]$ and $u, v\in V$,
	$$
	2( A(t,u) -A(t,v),
	u-v)_{(V^*, V)}
	+\|  B(t,u) -B(t,v) \|^2_{\call_2(l^2, H)}
	$$
	   \be \label{ma2pla p7}
	\le  -2^{1-p} C(n,p,s) \| u-v\|^p_{V_1}
	+ \left (   c_1+ 2
	\| \varphi_3(t)\|_{L^\infty(\o)}
 \right ) \| u-v \|^2_H,
\ee
	 where $\varphi_3\in L^1(0,T;
	L^\infty(\o))$.
	By \eqref{ma2pla p7},   we see  that
	 {\bf (H2)}  is satisfied
	 with the following parameters:
	 $g= c_1 + 2\| \varphi_3(\cdot)\|_{   
	L^\infty(\o) }$ and 
	 	   \be \label{ma2pla p7a}
	 q_1 =p, \ q_2 =q,\
	 q_3= 2,\
	 \alpha =0,\ \alpha_1 =0,\ 
	 \theta_j=   
	 \beta_{1, j}=\beta_{2,j} =0,
	 \quad  \ j=1,2,3.
\ee

		By   \eqref{sig77a}
		and Young's inequality,
		we find that there exists
		a constant $c_2=c_2
		(n, q,  s,   \o)>0$
		 such that  
		for $t\in [0,T]$
		and $v\in V$,
	\be\label{ma2pla p9a}
		\| B(t, v)\|^2_{\call_2(l^2, H)}
		\le  2 \sum_{i=1}^\infty 
		 \|\sigma_{1,i}  (t) \|^2_H
		 + c_2
		 + 2\sum_{i=1}^\infty \beta_i
		 \int_\o |v(x)|^q dx,
	\ee
		   which along with 	\eqref{sig2}
	shows that  	{\bf (H5)} is satisfied with
	the parameters:
	   \be \label{ma2pla p9}
	\gamma_{2,1} =
	0, \  \gamma_{2,2}=
	  2  \sum_{i=1}^\infty \beta_i,  \
   \gamma_{2,3} = 0.
\ee
 By  \eqref{ma2pla p7a},
	\eqref{ma1pla p8a} and \eqref{ma2pla p9}
we find that, 
in the present case, the constants in \eqref{h5c}
are given by
$$
\kappa_1 = 
\kappa_2=\kappa_3 =1,
$$
and  the condition \eqref{h5b}
reduces to:
 	   \be \label{ma2pla p10}
 \sum_{i=1}^\infty
\beta_i  < \delta_1,
\ee
which is given by \eqref{ma2pla 1}.
By \eqref{ma2pla p1a} , one can verify
$B$ is continuous from $V$ to $\call_2(U,H)$.
In addition,    since $p_1<q$, we find that $B$ 
  satisfies \eqref{h5d} and hence
  \eqref{h5c},
  which
  completes the proof. 
  \end{proof}

\subsection{Standard fractional 
 Laplace   equations with $p=2$}

In this subsection, we 
consider 
the standard fractional Laplace operator
$(-\Delta)^s$ which is a 
  special case
of  the fractional $p$-Laplace operator
when  $p=2$. 
In this case,  we prove  the existence and
uniqueness of solutions of the 
stochastic equation driven by transport noise.
More precisely, let
       $G: [0,T]
 \times V_1  \to \call_2 (l^2,H)$ be a 
  $(\calb ([0,T]) \times
  \calb (V_1), \calb (\call_2 (l^2,H)) )$-measurable
  function such that
  for all $t\in [0,T]$  and  $u, v \in V_1$,
 \be\label{G1}
 \| G(t, u)\|^2_{\call_2 (l^2,H)}
 \le \delta_4 \| u \|^2_{V_1}
 + \varphi_4 (t) (1+ \| u \|_H^2),
 \ee
 and
   \be\label{G2}
 \| G(t, u)-G(t,v)\|^2_{\call_2 (l^2,H)}
 \le  \delta_5 \| u-v \|^2_{V_1}
 + \varphi_4 (t)  \| u-v \|_H^2,
 \ee
  where $\delta_4 >0$
  and $\delta_5>0$ are  constants
  and $\varphi_4 \in L^1([0,T])$.

Note that \eqref{G2} implies that
for every $t\in [0,T]$,
$G (t,\cdot): V \to 
   \call_2(l^2,H)$   is continuous.
   We  further  assume   that
 if $u \in L^\infty(0,T; H)
\bigcap  (
\bigcap\limits_{1\le j\le J} L^{q_j}
(0,T; V_j)
  )$
  and 
  $\{u_n\}_{n=1}^\infty
 $ is  bounded   in
 $ L^\infty(0,T; H)
\bigcap  (
\bigcap\limits_{1\le j\le J} L^{q_j}
(0,T; V_j)
  )$
  such that
  $u_n  \to u$ in $L^1(0,T; H)$,
  then
  \be\label{G3}
   \lim_{n\to \infty}
  v^* G (\cdot ,   u_n) =  v^*G(\cdot,  u)
   \ \text{ in }  \  L^2(0,T; \call_2 (l^2,\R)),
   \quad   \forall  \  v\in {V},
   \ee
   where 
   $v^*$ is the element in $H^*$
   identified with $v$ in $H$  by  
   Riesz's  representation theorem.

    A  typical example of $G$ satisfying
    all conditions \eqref{G1}-\eqref{G3}
    is given below.

    \begin{example} 
 For every $i\in \N$, 
let 
      $g_i \in L^\infty(\R^n)
      \bigcap V_1$ such that
         \be \label{G4}
        \sum_{i=1}^\infty
        \left ( \| g_i\|^2_{L^\infty (\R^n)}
         + \|g_i\|^2_{V_1}
         \right )
         <\infty.
         \ee
      
  Given $u\in V_1
     $,  define an operator
  $G(u):  l^2\to H$ by
\be\label{G5}
  G (u) (a) (x)
  =\sum_{i=1}^\infty a_i 
  g_i (x) (-\Delta)^{\frac s2} u(x)  
   \quad \forall \ 
   a=\{a_i\}_{i=1}^\infty \in l^2,
   \ \ x\in \R^n.
\ee
It follows from 
  \eqref{G4}      that 
  $G (u): l^2\to H$ is a   Hilbert-Schmidt 
  operator 
   with norm
$$
 \|  G (u)   \|^2_{\call_2 (l^2, H)}
 =
 \sum_{i=1}^\infty
 \|  g_i (-\Delta)^{\frac s2} u   \|^2_H
 $$
   \be\label{G6}
 \le  \sum_{i=1}^\infty \| g_i \|^2_{L^\infty (\R^n)}
 \| (-\Delta)^{\frac s2} u   \|^2_H 
\le {\frac 12} C
 \sum_{i=1}^\infty \| g_i \|^2_{L^\infty (\R^n)}
 \|   u   \|^2_{V_1} ,
\ee
   where $C=C(n,2,s)$ is the number given by \eqref{constpl}
   with $p=2$.
 Similarly, by \eqref{G5}   we have, 
     for all 
 $u_1,u_2\in V_1$,
 \be\label{G7}
 \|G (  u_1)-
G ( u_2)\|^2_{\call_2(l^2,  H)}
  \le {\frac 12}C 
 \sum_{i=1}^\infty \| g_i \|^2_{L^\infty (\R^n)}
 \|   u_1-u_2   \|^2_{V_1}.
 \ee
  By \eqref{G6}-\eqref{G7} we see that
  \eqref{G1} and \eqref{G2} are fulfilled.
 
 It remains to show \eqref{G3}.
Suppose $u \in L^\infty(0,T; H)
\bigcap  
   L^{2}
(0,T; V_1 
  )$
  and 
  $\{u_n\}_{n=1}^\infty
 $ is  bounded   in
 $ L^\infty(0,T; H)
\bigcap    L^{2}
(0,T; V_1 
  )$
  such that
 \be\label{G7a}
 \lim_{n\to \infty}
 \| u_n -u\| _{L^1(0,T; H)} =0.
 \ee
 Since 
  $\{u_n\}_{n=1}^\infty
 $ is  bounded   in
 $ L^\infty(0,T; H)$, by \eqref{G7a}
 we get
 \be\label{G7b}
  \lim_{n\to \infty}
 \| u_n -u\| _{L^2(0,T; H)} =0.
 \ee
 Since 
 $\{u_n\}_{n=1}^\infty
 $ is  bounded   in 
$  L^{2}
(0,T; V_1 
  )$, there is  a constant
  $c_1>0$ such that
   \be\label{G7c}
 \sup_{n\in \N}
\| u_n\|^2_{L^2(0,T; V_1)}
  \le c_1.
  \ee
  For  every  $v\in V_1$, 
  by \eqref{G5} we have
$$
 \|v^* G (  u_1)-
v^* G ( u_2)\|^2_{\call_2(l^2,  \R )}
= \sum_{i=1}^\infty 
|(g_i (-\Delta)^{\frac s2}(u_1-u_2),   v)_H|^2
$$
\be\label{G8}
 =\sum_{i=1}^\infty 
|( (u_1-u_2),  (-\Delta)^{\frac s2}(g_i v))_H|^2
\le
\sum_{i=1}^\infty 
\|  u_1-u_2\|_H^2
 \|   (-\Delta)^{\frac s2}(g_i v)\|_H^2.
\ee
 Recall that for every $v\in C_0^\infty (\o)$,
  \be\label{G9}
 \|(-\Delta)^{\frac s2}(g_i v)\|^2_H
 =
 {\frac 12} C(n,2,s)
 \| g_i v\|^2_{V_1}
   \le
 C 
 \left (
 \|v\|^2_{L^\infty(\R^n)}
 \| g_i\|^2_{V_1}
 + \|g_i\|^2_{L^\infty (\R^n)}
 \| v\|^2_{V_1}
 \right ).
 \ee
 By \eqref{G8}-\eqref{G9} we get
for   all  $v\in  C _0^\infty(\o)$,  
$$
 \|v^* G (  u_1)-
v^* G ( u_2)\|^2_{\call_2(l^2,  \R)}
$$
$$
 \le C 
\left (
\sum_{i=1}^\infty
\| g_i \|^2_{V_1}
\| v \|^2_{L^\infty (\R^n)}
+
\sum_{i=1}^\infty \| g_i \| ^2_{L^\infty (\R^n)}
\| v \|^2_{V_1}
 \right )  \| u_1-u_2\|^2_H,
$$
which along with \eqref{G7b} shows that
\be\label{G10}
\lim_{n\to \infty} \int_0 ^T  \|v^* G (  u_n (t) )-
v^* G ( u(t))\|^2_{\call_2(l^2,  \R)} dt
=0, \quad
\forall \  v\in  C _0^\infty (\o).
\ee

Given $\zeta \in V_1$ and  $\varepsilon>0$,  
 since $C _0^\infty (\o)$ is dense
 in $V_1$, we find that
 there exists $v\in C _0^\infty (\o)$
 such that
 \be\label{G11}
 \| \zeta - v \|_{V_1} \le \varepsilon.
\ee
Note that 
$$
\int_0 ^T
 \|\zeta^* G (  u_n(t) )-
\zeta ^* G ( u(t) )\|^2_{\call_2(l^2,  \R)} dt
$$
\be\label{G12}
\le 2\int_0 ^T
\|(\zeta-v) ^*   (G (  u_n(t))-
  G ( u(t) )) \|^2_{\call_2(l^2,  \R)} dt
+ 2\int_0 ^T
\|v^* ( G (  u_n(t))-
 G ( u(t) ))\|^2_{\call_2(l^2,  \R)} dt.
\ee
For the first term on the right-hand side of
\eqref{G12}, by \eqref{pi}, \eqref{G7} and \eqref{G7c} we have
$$
2\int_0 ^T
\|(\zeta-v) ^*   (G (  u_n(t))-
  G ( u(t) )) \|^2_{\call_2(l^2,  \R)} dt
$$
$$
\le 
 2
\| \zeta-v\|_H^2
   \int_0 ^T
\|   G (  u_n(t))-
  G ( u(t)  ) \|^2_{\call_2(l^2,  H)} dt
 $$
 $$
\le
 C\| \zeta-v\|_H^2
 \sum_{i=1}^\infty \| g_i \|^2_{L^\infty (\R^n)}
  \int_0 ^T
\|     u_n(t) -
  u(t)    \|^2_{V_1}  dt
  $$ 
$$
\le 2
 C\| \zeta-v\|_H^2
 \sum_{i=1}^\infty \| g_i \|^2_{L^\infty (\R^n)}
\left (
\| u\|^2_{L^2(0,T; V_1)}
+ \sup_{n\in \N}
\| u_n\|^2_{L^2(0,T; V_1)}
 \right )
 $$
\be\label{G13}
\le 2 \lambda^{-1} 
 C\| \zeta-v\|_{V_1} ^2
 \sum_{i=1}^\infty \| g_i \|^2_{L^\infty (\R^n)}
\left (c_1+ 
\| u\|^2_{L^2(0,T; V_1)}
  \right ).
  \ee

By \eqref{G11}-\eqref{G13} we obtain that
for every  $\zeta \in V_1$ and  $\varepsilon>0$,  
  $$
\int_0 ^T
 \|\zeta^* G (  u_n(t) )-
\zeta ^* G ( u(t) )\|^2_{\call_2(l^2,  \R)} dt
$$
$$
\le 2 \varepsilon^2
\lambda^{-1} 
 C 
 \sum_{i=1}^\infty \| g_i \|^2_{L^\infty (\R^n)}
\left (c_1+ 
\| u\|^2_{L^2(0,T; V_1)}
  \right ).
$$
\be\label{G20}
 + 2\int_0 ^T
\|v^* ( G (  u_n(t))-
 G ( u(t) ))\|^2_{\call_2(l^2,  \R)} dt.
\ee
 Letting $n\to \infty$
  in \eqref{G20},
since $v\in C^\infty_0(\o)$, by \eqref{G10}
we get
for every  $\zeta \in V_1$ and $\varepsilon>0$,  
  \be\label{G21}
  \lim_{n\to \infty}
\int_0 ^T
 \|\zeta^* G (  u_n(t) )-
\zeta ^* G ( u(t) )\|^2_{\call_2(l^2,  \R)} dt
 \le 2 \varepsilon^2
\lambda^{-1} 
 C 
 \sum_{i=1}^\infty \| g_i \|^2_{L^\infty (\R^n)}
\left (c_1+ 
\| u\|^2_{L^2(0,T; V_1)}
  \right ).
 \ee
 Taking the limit of \eqref{G21} as $\varepsilon \to 0$,
  we obtain,
 for every  $\zeta \in V_1$,  
$$
  \lim_{n\to \infty}
\int_0 ^T
 \|\zeta^* G (  u_n(t) )-
\zeta ^* G ( u(t) )\|^2_{\call_2(l^2,  \R)} dt
=0,
$$
which yields \eqref{G3}, and thus completes the proof.
  \end{example}

We now consider the
fractional stochastic equation
 \eqref{lap1} with the diffusion term
  $B$ replaced by $B+G$,
  for which we have the following result.

\begin{thm}\label{ma3pla}
	If  \eqref{f2}-\eqref{sig5},
	  \eqref{ff1} and \eqref{G1}-\eqref{G2} hold 
  and
	\be\label{ma3pla 1} 
	2\le p_1 < q, \quad 
 \sum_{i=1}^\infty
\beta_i  < \delta_1, \quad
   \sum_{i=1}^\infty
	\gamma_i  \le   2 \delta_3,
	\quad  \delta_4<  C(n,2,s),
	\quad  \delta_5 \le {\frac 12} C(n,2,s),
	\ee
	where  $p_1$, $ \beta_i ,$
	and   $ \gamma_i  $
	are  the numbers  in 
	\eqref{sig3}-\eqref{sig4}, and
	$C(n,2,s)$ is the constant given by
	\eqref{constpl} with $p=2$.
Then for every $u_0 \in H$,
	system \eqref{lap1}-\eqref{lap3}
	with  $B$ replaced by $B+G$ 
	has a unique solution
	in the sense of Definition \ref{dsol}. 
	Moreover, the uniform estimates
	given by \eqref{main 1} are valid.
\end{thm}

\begin{proof}
We 
  only need to verify {\bf (H2)}
and {\bf (H5)} 
since {\bf (H1)}, {\bf (H3)} and {\bf (H4)}
are the same as   in Theorem \ref{ma2pla}.
 
By \eqref{ma2pla p7}
with $p=2$, \eqref{G2} and \eqref{ma3pla 1}
we get,
    for $ t\in [0,T]$ and $u, v\in V$,
	$$
	2( A(t,u) -A(t,v),
	u-v)_{(V^*, V)}
	+\|  B(t,u) -B(t,v) \|^2_{\call_2(l^2, H)}
	+\|  G(t,u) -G(t,v) \|^2_{\call_2(l^2, H)}
	$$
	   \be \label{ma3pla p7}
	\le   \left (   c_1+ \varphi_4 (t) +
	 2
	\| \varphi_3(t)\|_{L^\infty(\o)}
 \right ) \| u-v \|^2_H,
\ee
	 where $\varphi_3\in L^1(0,T;
	L^\infty(\o))$ and $\varphi_4\in L^1[0,T])$.
	 By \eqref{ma3pla p7},   we see  that
	 {\bf (H2)}  is satisfied
	 with the following parameters:
	 $g= c_1 +\varphi_4 +  2\| \varphi_3(\cdot)\|_{   
	L^\infty(\o) }$ and 
	 	   \be \label{ma3pla p7a}
	 q_1 =p, \ q_2 =q,\
	 q_3= 2,\
	 \alpha =0,\ \alpha_1 =0,\ 
	 \theta_j=   
	 \beta_{1, j}=\beta_{2,j} =0,
	 \quad  \ j=1,2,3.
\ee
	 	By \eqref{ma2pla p9a} and \eqref{G1}
		we find that
		for $t\in [0,T]$
		and $v\in V$,
		$$
		\| B(t, v)\|^2_{\call_2(l^2, H)}
		+\| G(t, v)\|^2_{\call_2(l^2, H)}
		$$
		$$
		\le  2 \sum_{i=1}^\infty 
		 \|\sigma_{1,i}  (t) \|^2_H
		 + c_2
		 + 2\sum_{i=1}^\infty \beta_i
		 \int_\o |v(x)|^q dx
		 +\delta_4 \| v \|^2_{V_1}
		 + \varphi_4 (t) (1+ \|v\|_H^2),
		 $$
		   which along with 	\eqref{sig2}
	shows that  	{\bf (H5)} is satisfied with
	the parameters:
	   \be \label{ma3pla p9}
	\gamma_{2,1} =
	\delta_4, \  \gamma_{2,2}=
	  2  \sum_{i=1}^\infty \beta_i,  \
   \gamma_{2,3} = 0.
\ee

By  
	\eqref{ma1pla p8a},  \eqref{ma3pla p7a}
	 and \eqref{ma3pla p9}
we find that, 
in the present case, the constants in \eqref{h5c}
are given by
$ 
\kappa_1 = 
\kappa_2=\kappa_3 =1,
$ 
and  the condition \eqref{h5b}
reduces to:
$$
 \sum_{i=1}^\infty
\beta_i  < \delta_1, \quad \delta_4 < C(n,2,s),
$$
which is implied  by \eqref{ma3pla 1}.

Since 
$B$ is continuous from $V$ to $\call_2(U,H)$ and
  satisfies  \eqref{h5c}, which along with \eqref{G3}
  implies that
  $B+G$ also  \eqref{h5c}. This completes the proof.
  \end{proof}

  \section{Appendix}

For reader's convenience,  we
  recall the following
Skorokhod-Jakubowski representation theorem
from \cite{brz1, jak2} on a topological space instead of
metric space.

\begin{prop}
\label{prop_sj}
 Suppose  $X$ 
 is  a topological space, and there
  exists a sequence
of continuous functions $g_n: X\rightarrow \mathbb{R}$ that separates points of $X$.
If    $\{\mu_n\}_{n=1}^\infty$
is a tight sequence of
   probability measures
  on $(X, \mathcal{B} (X) )$, then
  there exists a subsequence
  $\{\mu_{n_k}\}_{k=1}^\infty$
  of  $\{\mu_n\}_{n=1}^\infty$,
    a probability space $(
  \widetilde{
  \Omega},
  \widetilde{\mathcal{F}},
  \widetilde{\mathrm{P}})$, $X$-valued random variables
   $v_{ k}$  and $v$    such that
   the law of $v_{ k}$ is
   $\mu_{n_k}$ for all $k\in \mathbb{N}$ and
      $v_{ k} \rightarrow v$  $\widetilde{
  \mathrm{P}}$-almost surely
   as $k \rightarrow \infty$.
\end{prop}

\end{document}